\documentclass[12pt]{amsart}
\usepackage{amssymb,latexsym}
\usepackage{enumerate}

\makeatletter
\@namedef{subjclassname@2010}{%
  \textup{2010} Mathematics Subject Classification}
\makeatother

\newtheorem{theorem}{Theorem}[section]
\newtheorem{corollary}[theorem]{Corollary}
\newtheorem{lemma}[theorem]{Lemma}
\newtheorem{proposition}[theorem]{Proposition}
\theoremstyle{definition}
\newtheorem{definition}[theorem]{Definition}
\newtheorem{remark}[theorem]{Remark}

\numberwithin{equation}{section}

\usepackage{amssymb}
\usepackage[utf8]{inputenc}
\usepackage[T1]{fontenc}
\usepackage{courier}

\begin{document}

\baselineskip=17pt

\title[Estimation of Szlenk Index]{Estimation of the Szlenk index of Banach Spaces via Schreier spaces}

\author[R. Causey]{Ryan Causey}
\address{Department of Mathematics \\ Texas A\&M University \\
College Station, TX 77845}
\email{rcausey@math.tamu.edu}

\date{}

\begin{abstract}
For each ordinal $\alpha<\omega_1$, we prove the existence of a Banach space with a basis and Szlenk index $\omega^{\alpha+1}$ which is universal for the class of separable Banach spaces with Szlenk index not exceeding $\omega^\alpha$.  Our proof involves developing a characterization of which Banach spaces embed into spaces with an FDD with upper Schreier space estimates.  
\end{abstract}



\maketitle

\section{Introduction}

Two types of questions have long been of significance in Banach space theory:  Those of universality and those of coordinatization.  One early result which answers a question of each type is that of the universality of $C[0,1]$ for the class of all separable Banach spaces.  This result also affirmatively answers the question of whether any separable Banach space can be embedded in a space with a basis.  Other questions of coordinatization which naturally follow this one include determining when one can embed a particular type of Banach space, such as a reflexive space or an Asplund space, into a Banach space with a coordinate system which has the same or related properties.   

Two other important results concerning universality are those of Pe\l czy\'{n}ski \cite{P}, who showed that there exist Banach spaces $X, X_u$ with a basis and an unconditional basis, respectively, so that if $Y$ is any Banach space with a basis (respectively, unconditional basis), then $Y$ embeds complamentably in $X$ (respectively $X_u$).  In fact, the basis of $Y$ is equivalent to a subsequence of the basis of $X$ (respectively, the basis of $X_u$), and the closed span of this subsequence is a complemented subspace of $X$, (respectively, $X_u$).  Some early major results concerning coordinatization are those of Zippin \cite{Z}, who showed that any separable reflexive space may be embedded into a space with shrinking and boundedly-complete basis, and any space with separable dual can be embedded into a space with a shrinking basis.  It is no coincidence that we have linked these two types of questions here.  The power of bases and other coordinate systems can greatly simplify embedding and universality questions.  For example, Schechtman's space $W$, which has a finite dimensional decomposition and the property that any space with a finite dimensional decomposition embeds almost isometrically into $W$ \cite{Sch}, was used to construct universal spaces in \cite{OSZ2},\cite{FOSZ}.   The technique we use for making questions of universality more tractable will be to embed a space with certain properties into a space with FDD with the same or related properties.  

Another tool used in the study of universality is the Szlenk index.  With it, Szlenk \cite{SZLENK} answered in the negative whether there exists a separable, reflexive space which is universal for the class of all separable, reflexive Banach spaces.  Since then, this and other ordinal indices have seen fruitful use in Banach space theory.  It was shown to completely characterize up to isomorphism separable $C(K)$ spaces \cite{La}.  It was shown by Odell, Schlumprecht, and Zs\'{a}k that Tsirelson spaces act as a sort of upper envelope, via subsequential tree estimates, for certain classes of Banach spaces with bounded Szlenk index \cite{OSZ2}.  Tree estimates were shown to be the uncoordinatized version of the notion of block estimates.  In section $2$, we define the relevant notions to relate the results concerning tree and block estimates.   

Filling a role similar to that played by the Tsirelson spaces are the Schreier spaces.  In section $3$ we will define for each ordinal $\alpha<\omega_1$ the Schreier family $S_\alpha$, a family of subsets of the natural numbers, and then use the family $S_\alpha$ to define the Schreier space $X_\alpha$ and deduce some facts about them.  The proofs of our main theorems are presented in section $5$.  In that section, we begin by observing that we can weaken slightly the hypotheses of a theorem from \cite{FOSZ} connecting tree estimates to block estimates.  We then establish the following connection between Szlenk index and Schreier space estimates.  

\begin{theorem}
If $X$ is a Banach space and $\alpha$ is a countable ordinal with $Sz(X)\leq \omega^\alpha$, then $X$ satisfies subsequential $X_\alpha$-upper tree estimates.  \end{theorem}

We conclude that section by proving a universality result.    

\begin{theorem} For every countable ordinal $\alpha$, there exists a Banach space $Z$ with FDD $E=(E_n)$ satisfying subsequential $X_\alpha$-upper block estimates such that if $X$ is any separable Banach space with $Sz(X)\leq \omega^\alpha$, then $X$ embeds into $Z$. \end{theorem}

Combining this theorem with a result of Johnson, Rosenthal, and Zippin and a result of Odell, Schlumprecht, and Zs\'{a}k, we can deduce the following.  

\begin{corollary} For every countable ordinal $\alpha$, there exists a Banach space $W$ with a basis and $Sz(W)\leq \omega^{\alpha+1}$ such that if $X$ is any separable Banach space with $Sz(X)\leq \omega^\alpha$ then $X$ embeds into $W$.    
\end{corollary}

This is a strengthening of a theorem from \cite{FOSZ}, which proved the above in the case that $\alpha=\beta \omega$ for some $\beta<\omega_1$.  

This paper was completed at Texas A\&M under the direction of Thomas Schlumprecht as part of the author's doctoral dissertation.  The author thanks Dr. Schlumprecht for his insights and direction during its completion.

\section{Definitions and Notation}

Throughout, unless otherwise stated, Banach spaces are real, separable, and infinite-dimensional.  

A sequence $(E_n)$ of finite dimensional spaces is called a \emph{finite dimensional decomposition} (FDD) for a Banach space $Z$ if for each $z\in Z$ there exists a unique sequence $(z_n)$ so that $z_n\in E_n$ and $z=\displaystyle\sum_{n=1}^\infty z_n$.  Let $Z$ be a Banach space with an FDD $E=(E_n)$ and $n\in \mathbb{N}$.  We let $P^E_n$ denote the $\emph{n}$-th \emph{coordinate projection} $P^E_n:Z\to E_n$ defined by $\displaystyle\sum z_i\mapsto z_n$, where $z_i\in E_i$ for all $i\in \mathbb{N}$.  For $z\in Z$, we define $\text{supp}_E z=\{n: P^E_nz\neq 0\}$.  If no confusion is possible, we may write $\text{supp\ }z$ for $\text{supp}_E z$.  For $A\subset \mathbb{N}$ finite, $P^E_A=\displaystyle\sum_{n\in A}P^E_n$.  The \emph{projection constant} $K(E,Z)$ of $(E_n)$ is defined by $$K=K(E,Z)=\underset{m\leq n}{\sup}\|P^E_{[m,n]}\|,$$  By the Principle of Uniform Boundedness, $K$ is finite.  We call an FDD $E$ for $Z$ bimonotone if $K(E,Z)=1$.  If a space $Z$ has an FDD $E$, one can always endow $Z$ with an equivalent norm which makes $E$ a bimonotone FDD for $Z$.    

A sequence $(F_n)$ is a blocking of $(E_n)$ if there exists $1=m_0<m_1<\ldots$ so that $F_n=\underset{j=m_{n-1}}{\overset{m_n-1}{\bigoplus}}E_j$ for all $n\in \mathbb{N}$.  If $(E_n)$ is an FDD for a Banach space $Z$, then the blocking $(F_n)$ is an FDD for $Z$ with projection constant not exceeding that of $(E_n)$.  

For any sequence $(E_n)$ of finite-dimensional spaces, we let $$c_{00}\Bigl(\underset{n=1}{\overset{\infty}{\bigoplus}}E_n\Bigr)=\bigl\{(z_n):z_n\in E_n \forall n\in \mathbb{N}, \{i\in \mathbb{N}:z_n\neq 0\} \text{ is finite\ }\bigr\}.$$

This space is dense in any Banach space for which $(E_n)$ is an FDD.   

If $Z$ is a Banach space with FDD $(E_n)$, we let $Z^{(*)}$ denote the closure of $c_{00}\Bigl(\underset{n=1}{\overset{\infty}{\bigoplus}}E^*_n\Bigr)$ in $Z^*$.  We call an FDD $(E_n)$ for a Banach space $Z$ \emph{shrinking} if $Z^*=Z^{(*)}$.  It is not necessarily true that the embedding $E_n^*\hookrightarrow Z^*$ is isometric unless $(E_n)$ is bimonotone.  The norm on $E^*_n$ is that induced by $Z^*$, and not the norm it inherits as the dual of $E_n$.  If $(E_n)$ is bimonotone, $Z^{(*)(*)}=Z$.  

An FDD $(E_n)$ for a Banach space $Z$ is called \emph{boundedly-complete} if whenever $z_n\in E_n$ and $\underset{n\in \mathbb{N}}{\sup}\Bigl\|\displaystyle\sum_{n=1}^\infty z_n\Bigr\|<\infty$, $\displaystyle\sum_{n=1}^n z_n$ converges in $Z$.  Any space with a boundedly-complete FDD is naturally a dual space.

A sequence (finite or infinite) of finitely supported non-zero vectors $(z_n)$ so that $$\max \text{supp}_E z_n<\min \text{supp}_E z_{n+1}$$

for all appropriate $n$ is called a \emph{block sequence} with respect to $E$.  When no confusion is possible, we simply call $(z_n)$ a block sequence.   

Throughout, we will through an abuse of notation conflate a basis for a Banach space $(e_n)$ with the corresponding FDD in which each finite dimensional space is the span of the corresponding basis vector.

\begin{definition}

If $(e_n)$ and $(f_n)$ are sequences in some Banach spaces and $C>0$ is such that $$\Bigl\|\displaystyle\sum a_nc_n\Bigr\|\leq C\Bigl\|\displaystyle\sum a_nf_n\Bigr\|$$
for all $(a_n)\in c_{00}$, then we say $(e_n)$ is $C$-\emph{dominated} by $(f_n)$, or that $(f_n)$ $C$-\emph{dominates} $(e_n)$.  We say $(e_n),(f_n)$ are $C$-\emph{equivalent} if there exist constants $A,B>0$ so that $AB\leq C$, $(e_n)$ $A$-dominates $(f_n)$, and $(f_n)$ $B$-dominates $(e_n)$.  

We say $(f_n)$ \emph{dominates} $(e_n)$ or $(e_n)$ is \emph{dominated by} $(f_n)$ if there is some $C>0$ so that $(f_n)$ $C$-dominates $(e_n)$.  We say $(e_n)$ and $(f_n)$ are \emph{equivalent} if there exists some $C>0$ so that they are $C$-equivalent.  \end{definition}

\begin{definition}

Let $V$ be a Banach space with normalized, $1$-unconditional basis $(v_n)$.  Then $(v_n)$ is $C$-\emph{right dominant} (respectively $C$-\emph{left dominant }) if for subsequences $(k_n),(\ell_n)$ of $\mathbb{N}$ so that $k_n\leq \ell_n$ for each $n$, $(v_{k_n})$ is $C$-dominated by (respectively $C$-dominates) $(v_{\ell_n})$.  We say $(v_n)$ is \emph{right dominant} (respectively \emph{left dominant}) if it is $C$-right dominant (respectively $C$-left dominant) for some $C\geq 1$.

 We say that $(v_n)$ is $C$-\emph{block-stable} if whenever $(x_n),(y_n)$ are normalized block sequences in $V$ with $$\max\bigl( \text{supp}_V (x_n)\cup \text{supp}_V (y_n)\bigr)<\min \bigl(\text{supp}_V (x_{n+1})\cup \text{supp}_V (y_{n+1})\bigr),$$  

then $(x_n)$ and $(y_n)$ are $C$-equivalent.  We say $(v_n)$ is \emph{block-stable} if it is $C$-block-stable for some $C$.  
\end{definition}

\begin{definition}

Let $Z$ be a Banach space with an FDD $E=(E_n)$, $V$ a Banach space with a normalized, $1$-unconditional basis $(v_n)$, and let $1\leq C<\infty$.  We say $E$ \emph{satisfies subsequential} $C$-$V$- \emph{upper block estimates} in $Z$ if any normalized block sequence $(z_n)$ in $Z$ is $C$-dominated by $(v_{m_n})$, where $m_n=\min \text{supp}_E z_n$.  We say $E$ \emph{satisfies subsequential} $C$-$V$-\emph{lower block estimates} in $Z$ if any normalized block sequence $(z_n)$ $C$-dominates $(e_{m_n})$.  We say $E$ \emph{satisfies subsequential} $V$-\emph{upper} (respectively \emph{lower}) \emph{block estimates} in $Z$ if there is some $C$ such that $E$ satisfies $C$-$V$-upper (respectively lower) block estimates in $Z$.  

\end{definition}

A standard perturbation argument gives the following, which allows flexibility in choosing the indices for norm estimates of block sequences.

\begin{proposition}
Let $V$ be a Banach space with normalized, $1$-unconditional basis $(v_n)$, and let $Z$ be a Banach space with FDD $(E_n)$ which satisfies subsequential $C$-$V$-upper (respectively, lower) block estimates in $Z$.  Then if $(x_n)$ is a normalized block sequence in $E$ and $(k_n)$ is a subsequence of $\mathbb{N}$ with $\max \text{\rm \ supp\ }x_n<k_{n+1}\leq \min \text{\rm \ supp\ } x_{n+1}$ for all $n$, then $(x_n)$ is $C$-dominated by (respectively, $C$-dominates) $(v_{k_n})$.   \end{proposition}

Next, we define the uncoordinatized version of block estimates, which was first considered in \cite{OSZ1}.  

\begin{definition}

For $\ell\in \mathbb{N}$, we define $$T_\ell=\{(n_1,\ldots,n_\ell): n_1<\ldots<n_\ell, n_i\in \mathbb{N}\}$$

and $$T_\infty=\underset{\ell=1}{\overset{\infty}{\bigcup}}T_\ell,\text{\ \ \ \ \ \ \ } T_\infty^\text{even}=\underset{\ell=1}{\overset{\infty}{\bigcup}}T_{2\ell}.$$

An \emph{even tree} in a Banach space $X$ is a family $(x_t)_{t\in T^\text{even}_\infty}$ in $X$.  Sequences of the form $(x_{(t,k)})_{k>k_{2n-1}}$, where $n\in \mathbb{N}$ and $t=(k_1,\ldots,k_{2n-1})\in T_\infty$, are called \emph{nodes}.  A sequence of the form $(k_{2n-1},x_{(k_1,\ldots,k_{2n})})_{n=1}^\infty$, with $k_1<k_2<\ldots$, is called a \emph{branch} of the tree.  An even tree is called \emph{weakly null} if every node is a weakly null sequence.  If $X$ is a dual space, an even tree is called $w^*$ \emph{null} if every node is $w^*$ null.  If $X$ has an FDD $E=(E_n)$, a tree is called a \emph{block even tree of} $E$ if every node is a block sequence of $E$.  

If $T\subset T^\text{even}_\infty$ is closed under taking restrictions so that for each $t\in T\cup \{\varnothing\}$ and for each $m\in \mathbb{N}$ the set $\{n\in \mathbb{N}: (t,m,n)\in T\}$ is either empty or infinite, and if the latter occurs for infinitely many values of $m$, then we call $(x_t)_{t\in T}$ a \emph{full subtree}.  Such a tree can be relabeled to a family indexed by $T^\text{even}_\infty$ and the branches of $(x_t)_{t\in T}$ are branches of $(x_t)_{t\in T^\text{even}_\infty}$ and that the nodes of $(x_t)_{t\in T}$ are subsequences of the nodes of $(x_t)_{t\in T^\text{even}_\infty}$. 

\end{definition}

\begin{definition}

Let $V$ be a Banach space with normalized, $1$-unconditional basis $(v_n)$ and $C\geq 1$.  Let $X$ be an infinite-dimensional Banach space.  We say that $X$ \emph{satisfies subsequential} $C$-$V$-\emph{lower tree estimates} if every normalized, weakly null even tree $(x_t)_{t\in T^\text{even}_\infty}$ in $X$ has a branch $(k_{2n-1},x_{(k_1, \ldots,k_{2n})})$ so that $(x_{(k_1, \ldots, k_{2n})})_n$ $C$-dominates $(v_{k_{2n-1}})_n$.  We say $X$ \emph{satisfies subsequential} $C$-$V$-\emph{upper tree estimates} if every normalized, weakly null even tree $(x_t)_{t\in T^\text{even}_\infty}$ in $X$ has a branch $(k_{2n-1},x_{(k_1,\ldots,k_{2n})})$ so that $(x_{(k_1, \ldots, k_{2n})})_n$ is $C$-dominated by $(v_{k_{2n-1}})$.  

We say that $X$ \emph{satisfies subsequential} $V$-\emph{upper (respectively lower) tree estimates} if it satisfies $C$-$V$-upper (respectively lower) tree estimates for some $C\geq 1$.  

If $X$ is a subspace of a dual space, we say that $X$ \emph{satisfies subsequential $C$-$V$-lower $w^*$ tree estimates} if every $w^*$ null even tree $(x_t)_{t\in T^\text{even}_{\infty}}$ in $X$ has a branch $(x_{(n_1,\ldots,n_{2i})})_{i=1}^\infty$ which $C$-dominates $(v_{n_{2i-1}})$.  \end{definition} 

\vspace{5mm}

For $C\geq 1$, let $\mathcal{A}_V(C)$ be the class of Banach spaces which satisfy subsequential $C$-$V$-upper tree estimates, and $\mathcal{A}_V=\underset{C\geq 1}{\bigcup}\mathcal{A}_V(C)$.  We prove in Section $5$ that this class has a universal element.  That is, it contains an element into which any other element of this class embeds.  

\vspace{5mm}

The upper and lower estimates are dual notions in a very natural way.  We make this precise below.    

\begin{proposition}\cite[Lemma 3]{OSZ1}
If $Z$ is a Banach space with FDD $(E_n)$, and $V$ is a Banach space with normalized, $1$-unconditional basis $(v_n)$, then the following are equivalent:  \begin{enumerate}
\item $(E_n)$ satisfies subsequential $V$-upper block estimates in $Z$.

\item $(E_n^*)$ satisfies subsequential $V^{(*)}$-lower block estimates in $Z^{(*)}$.  \end{enumerate} \end{proposition}

\begin{lemma}\cite[Lemma 2.7]{FOSZ}
Let $X$ be a Banach space with separable dual, and let $V=(v_n)$ be a normalized, $1$-unconditional, right dominant basis.  If $X$ satisfies subsequential $V$-upper tree estimates, then $X^*$ satisfies subsequential $V^{(*)}$-lower $w^*$ tree estimates.  \end{lemma}

We will, using established embedding theorems and a particular method of constructing new Banach spaces with FDDs from old, find spaces with FDDs with the desired properties.  Our usual space for doing so will be the space $Z^V$.  

\begin{definition} If $Z$ is a Banach space with FDD $E=(E_n)$ and $V$ is a Banach space with normalized, $1$-unconditional basis $(v_n)$, we define a new norm on $c_{00}\Bigl(\underset{n=1}{\overset{\infty}{\bigoplus}}E_n\Bigr)$ by $$\|z\|_{Z^V}=\max \Bigl\{\Bigl\|\displaystyle\sum_{i=1}^n \|P^E_{[m_{i-1},m_i)}z\|_Z v_{m_{i-1}}\Bigr\|_V: 1\leq m_0<\ldots<m_n, n,m_i\in \mathbb{N}\Bigr\}.$$

We then let $Z^V$ be the completion of $c_{00}\Bigl(\underset{n=1}{\overset{\infty}{\bigoplus}}E_n\Bigr)$ under the norm $\|\cdot\|_{Z^V}$.  Then $(E_n)$ is an FDD for $Z^V$ with $K(E,Z^V)\leq K(E,Z)$.  We connect some properties of the FDD $(E_n)$ for $Z^V$ and the basis $(v_n)$ of $V$.  

\end{definition}

\begin{proposition}\cite[Corollary 7, Lemma 8]{OSZ1}

Let $V$ be a Banach space with a normalized, $1$-unconditional basis $(v_n)$, and $Z$ a space with FDD $(E_n)$.  \begin{enumerate}
\item If $(v_n)$ is boundedly-complete, $(E_n)$ is a boundedly-complete FDD for $Z^V$.
\item If $(v_n)$ is a shrinking basis for $V$ and $(E_n)$ is a shrinking FDD for $Z$, then $(E_n)$ is a shrinking FDD for $Z^V$.  \end{enumerate} \end{proposition}

We conclude this section with generalizations of Lemmas $2$ and $10$ from \cite{OSZ1}.  The proofs are very similar, but we include them for completeness.  The difference is that we do not assume block stability of $(v_n)$, but only that $(v_n)$ satisfies lower block estimates in itself.  This means that for a normalized block sequence $(x_n)$ in $V$ with $\min\text{supp\ } x_n=m_n$, there is some $C\geq 1$ so that $(x_n)$ $C$-dominates $(v_{m_n})$.  We will use the following abbreviation.  If $M=(m_n)$ is a subsequence of $\mathbb{N}$, and $(v_n)$ is a basis for the Banach space $V$, we let $V_M$ denote the closed linear span of $(v_{m_n})$ in $V$.  

\begin{lemma}
If $V$ is a Banach space with normalized, $1$-unconditional basis $(v_n)$ which satisfies subsequential $C$-$V$-lower block estimates in itself, and $Z$ is a space with FDD $(E_n)$, then $(E_n)$ satisfies subsequential $2C$-$V$-lower block estimates in $Z^V$.  \end{lemma}

\begin{proof}
Fix a normalized block sequence $(z_n)$ in $Z^V$ and $(a_n)\in c_{00}$.  Let $m_n=\min \text{supp} z_n$.  For each $n\in \mathbb{N}$, fix an increasing sequence of natural numbers $(k^{(n)}_i)_{i=0}^{\ell_n}$ so that $$1=\Bigl\|\displaystyle\sum_{i=1}^{\ell_n} \|P^E_{[k_{i-1}^{(n)},k_i^{(n)})}z_n\|_Z v_{k_{i-1}^{(n)}}\Bigr\|_V.$$

Because the basis $(v_n)$ is bimonotone, we can assume that $k_0^{(n)}\leq m_n<k_1^{(n)}$ and $k^{(n)}_{\ell_n}=m_{n+1}$.  For each $n$ and $1\leq i \leq n_\ell$, put $m_i^{(n)}=k^{(n)}_i$.  Put $m^{(n)}_0=m_n$.  Then because $k^{(n)}_i=m^{(n)}_i$ for each $i>0$, we get $$\Bigl\|\displaystyle\sum_{i=2}^{\ell_n} \|P^E_{[m^{(n)}_{i-1},m^{(n)}_i)}z_n\|_Z v_{m_{i-1}^{(n)}}\Bigr\|_V=\Bigl\|\displaystyle\sum_{i=2}^{\ell_n} \|P^E_{[k^{(n)}_{i-1},k^{(n)}_i)}z_n\|_Z v_{m_{i-1}^{(n)}}\Bigr\|_V.$$

Using the triangle inequality and noting that $P^E_{[m^{(n)}_0,m^{(n)}_1)}z_n=P^E_{[k^{(n)}_0,k^{(n)}_1)}z_n$, we see that \begin{align*} 1&\leq 2\max \Bigl\{\|P^E_{[m^{(n)}_0,m^{(n)}_1)}z_n\|_Z, \Bigl\|\displaystyle\sum_{i=2}^{\ell_n} \|P^E_{[k^{(n)}_{i-1},k^{(n)}_i)}z_n\|_Z v_{m_{i-1}^{(n)}}\Bigr\|_V\Bigr\} \\ &\leq 2\Bigl\|\displaystyle\sum_{i=1}^{\ell_n} \|P^E_{[m^{(n)}_{i-1}, m^{(n)}_i)}z_n\|_Z v_{m_{i-1}^{(n)}}\Bigr\|_V.\end{align*}

Let $y_n=\displaystyle\sum_{i=1}^{\ell_n} \|P^E_{[m^{(n)}_{i-1}, m^{(n)}_i)}z_n\|_Z v_{m_{i-1}^{(n)}}$.  We note that $\min \text{supp} y_n=m_n$.  We have already shown that $\|y_n\|\geq \frac{1}{2}$.  Let $(a_i)\in c_{00}$ and let $(k_i)_{i=0}^\ell$ be the concatenation of the sequences $(m_i^{(n)})_{i=0}^{\ell_n}$ for each $1\leq n\leq M=\max \text{supp} (a_j)$.  For $z=\displaystyle\sum_{i=1}^\infty a_nz_n$, we get that 

\begin{align*} \|z\|_{Z^V} & \geq \Bigl\|\displaystyle\sum_{i=1}^\ell \|P^E_{[k_{i-1},k_i)}z\|_Z v_{k_{i-1}}\Bigr\|_V \\ &  =\Bigl\|\displaystyle\sum_{n=1}^\infty \displaystyle\sum_{i=1}^{\ell_n} a_n\|P^E_{[m^{(n)}_{i-1}, m^{(n)}_i)}z_n\|_Z v_{m^{(n)}_{i-1}}\Bigr\|_V \\ & =\Bigl\|\displaystyle\sum_{n=1}^\infty a_n y_n\Bigr\|_V\geq \frac{1}{C} \Bigl\|\displaystyle\sum_{n=1}^\infty a_n\|y_n\|v_{m_n}\Bigr\|_V\geq \frac{1}{2C}\Bigl\|\displaystyle\sum_{n=1}^\infty a_n v_{m_n}\Bigr\|_V.\end{align*}

This gives the result.

\end{proof}

\begin{remark} We cannot omit the initial part of the proof above in which we pass from the $k^{(n)}_i$ to the $m^{(n)}_i$.  That is, we could not have assumed that $k^{(n)}_0=\min \text{supp} z_n$ and $1=\Bigl\|\displaystyle\sum_{i=1}^{\ell_n} \|P^E_{[k_{i-1}^{(n)},k_i^{(n)})}z_n\|_Z v_{k_{i-1}^{(n)}}\Bigr\|_V$.  In fact, the factor of $2$ which occurs above is sharp.  

To see this, let $V=\mathbb{R}\oplus_1 c_0$ and let $(v_n)$ denote the natural basis for $V$.  Then $(v_n)$ is normalized and$1$-unconditional.  If we let $(e_n)$ denote the canonical $c_0$ basis, and let $z_n=\frac{1}{2}e_{2n}+\frac{1}{2}e_{2n+1}$, $(z_n)$ is a normalized block sequence in $c_0^V$.  To see this, observe that $$\Bigl\|\|P_{[1,2n+1)}z_n\|_{c_0}v_1+\|P_{[2n+1,2n+2)}z_n\|_{c_0}v_{2n+1}\Bigr\|_V=\frac{1}{2}+\frac{1}{2}=1.$$

But if $2n\leq k_0<k_1<\ldots<k_\ell$, then \begin{align*} \Bigl\|\displaystyle\sum_{i=1}^{\ell} \|P_{[k_{i-1},k_i)}z_n\|_{c_0} v_{k_{i-1}}\Bigr\|_V=\Bigl\|\displaystyle\sum_{i=1}^{\ell} \|P_{[k_{i-1},k_i)}z_n\|_{c_0} v_{k_{i-1}}\Bigr\|_{c_0}\leq \frac{1}{2}.\end{align*}

\end{remark}

\begin{lemma}
Let $V$ be a Banach space with normalized, $1$-unconditional basis $(v_n)$ which satisfies subsequential $V$-lower block estimates in $V$. If $M=(m_n)$ is a subsequence of $\mathbb{N}$ and $Z$ is a space with FDD $E=(E_n)$ satisfying subsequential $V_M$-lower block estimates in $Z$, then $W=Z\oplus_\infty V_{\mathbb{N}\setminus M}$ has an FDD satisfying subsequential $V$-lower block estimates in $W$.  \end{lemma}

\begin{proof}

Let $C$ be such that $(v_n)$ satisfies subsequential $C$-$V$-lower block estimates in $V$ and such that $E$ satisfies subsequential $C$-$V_M$-lower block estimates in $Z$.  

We define an FDD $F=(F_n)$ of $W$ by \begin{displaymath}
   F_n= \left\{
     \begin{array}{rr}
       \text{span}(v_n) & : n\notin M,\\
       E_k & : n=m_k. 
     \end{array}
   \right.
\end{displaymath} 

Let $P$ and $Q$ be the projections onto $Z$ and $V_{\mathbb{N}\setminus M}$, respectively.  Let $(z_n)$ be a normalized block sequence in $W$, and let $b_n=\min \text{supp}_F z_n$.  Let $x_n=Pz_n$, $y_n=Qz_n$.  Let $N_1=\{n: x_n\neq 0\}$, $N_2=\{n: y_n\neq 0\}$.  We note that $(y_n)_{n\in N_2}$ is a block sequence in $V$ with $b_n\leq \min \text{supp}_V y_n<b_{n+1}$ for each $n\in N_2$.  Applying Proposition $2.4$, we get that 

$$\Bigl\|\displaystyle\sum_{n\in N_2} a_ny_n\Bigr\|_V\geq C^{-1}\Bigl\|\displaystyle\sum_{n\in N_2} a_n \|y_n\|v_{b_n}\Bigr\|_V.$$

Next, we note that $(x_n)_{n\in N_1}$ is a block sequence in $E$.  For $n\in N_1$, let $p_n=\min \text{supp}_E x_n$.  Unravelling the definition of $V_M$-lower block estimates in $Z$ gives that $$\Bigl\|\displaystyle\sum_{n\in N_1} a_n x_n\Bigr\|_Z\geq C^{-1}\Bigl\|\displaystyle\sum_{n\in N_1}a_n \|x_n\|_Z v_{m_{p_n}}\Bigr\|_V.$$

We note that, by construction, $b_n\leq m_{p_n}<b_{n+1}$.  Applying Proposition $2.4$ to $(v_{b_n})_{n\in N_1}$ and $(v_{m_{p_n}})_{n\in N_1}$ gives that 

$$\Bigl\|\displaystyle\sum_{n\in N_1} a_n\|x_n\|_Z v_{m_{p_n}}\Bigr\|_V\geq C^{-1}\Bigl\|\displaystyle\sum_{n\in N_1} a_n \|x_n\|_Zv_{b_n}\Bigr\|_V.$$

Letting $z=\displaystyle\sum_{n=1}^\infty a_nz_n$, we get that \begin{align*} \Bigl\|\displaystyle\sum_{n=1}^\infty a_n v_{b_n}\Bigr\|_V & \leq \Bigl\|\displaystyle\sum_{n=1}^\infty a_n(\|x_n\|_Z+\|y_n\|_V)v_{b_n}\Bigr\|_V \\ & \leq \Bigl\|\displaystyle\sum_{n\in N_1} a_n \|x_n\|_Zv_{b_n}\Bigr\|_V+\Bigl\|\displaystyle\sum_{n\in N_2} a_n\|y_n\|_V v_{b_n}\Bigr\|_V \\ & \leq 2\max \Bigl\{\Bigl\|\displaystyle\sum_{n\in N_1} a_n\|x_n\|_Z v_{b_n}\Bigr\|_V,\Bigl\|\displaystyle\sum_{n\in N_2} a_n\|y_n\|_V v_{b_n}\Bigr\|_V\Bigr\} \\ & \leq 2 \max \Bigl\{C^2\Bigl\|\displaystyle\sum_{n\in N_1} a_nx_n\Bigr\|_Z,C \Bigl\|\displaystyle\sum_{n\in N_2} a_ny_n\Bigr\|_V\Bigr\}\\ &=2\max \{C^2\|Pz\|_Z,C\|Qz\|_V\}\leq 2C^2 \|z\|_W.\end{align*}

\end{proof}

We collect a fact from \cite{FOSZ} relating the concept of infinite games to trees and branches.  For more information about these infinite games, see \cite{OS1}.  First we must recall some of their notation.  If $X$ is a Banach space, $A\subset[\mathbb{N}\times S_X]^\omega$, and $\varepsilon \in(0,1)$, we let $$A_\varepsilon = \Bigl\{(k_n,y_n)\in [\mathbb{N}\times S_X]^\omega: \exists (\ell_n,x_n)\in A \text{\ \ }\ell_n\leq k_n, \|x_n-y_n\|<\varepsilon2^{-n}\text{\ }\forall n\in \mathbb{N}\Bigr\}.$$

In the following proposition, the closure $\overline{A}_\epsilon$ is with respect to the product topology on $[\mathbb{N}\times S_X]^\omega$.  For it, we also need the following definition.  \begin{definition} Let $E=(E_n)$ be an FDD for a Banach space $X$ and let $\overline{\delta}=(\delta_n)$ with $\delta_n\downarrow 0$.  A sequence $(x_n)\subset S_X$ is called a $\overline{\delta}$-\emph{skipped block w.r.t.} $(E_n)$ if there exist integers $1=k_0<k_1<\ldots$ so that for all $n\in \mathbb{N}$, $$\|P^E_{(k_{n-1}, k_n)}y_n - y_n\|<\delta_n.$$  \end{definition}

\begin{proposition}\cite[Proposition 2.6]{FOSZ} Let $X$ be an infinite-dimensional closed subspace of a dual space $Z$ with boundedly-complete FDD $(E_n)$.  Let $A\subset [\mathbb{N}\times S_X]^\omega$.  The following are equivalent.  

\begin{enumerate} \item For all $\varepsilon>0$ there exists $(K_n)\subset \mathbb{N}$ with $K_1<K_2<\ldots$, $\overline{\delta}=(\delta_n)\subset (0,1)$ with $\delta_n\downarrow 0$ and a blocking $F=(F_n)$ of $(E_n)$ such that if $(x_n)\subset S_X$ is a $\overline{\delta}$-skipped block sequence of $(F_n)$ in $Z$ with $\|x_n-P^F_{(r_{n-1},r_n)}\|<\delta_n$ for all $n\in \mathbb{N}$, where $1\leq r_0<r_1\ldots$, then $(K_{r_{n-1}}, x_n)\in \overline{A}_\varepsilon$.

\item for all $\varepsilon>0$, every normalized $w^*$ null even tree in $X$ has a branch in $\overline{A}_\varepsilon$.  \end{enumerate}\end{proposition}

\section{Schreier Families, Schreier Spaces}

Throughout, we will assume subsets of $\mathbb{N}$ are written in increasing order.  Let $[\mathbb{N}]^{<\omega}$ denote the set of all finite subsets of $\mathbb{N}$, and $[\mathbb{N}]$ the set of all infinite subsets of $\mathbb{N}$.  We associate a set $F$ with the function $1_F\in \{0,1\}^\mathbb{N}$ and consider this space with the product topology.  We consider the families $[\mathbb{N}],[\mathbb{N}]^{<\omega}$ as being ordered by extension.  That is, the predecessors of an element are its initial segments.  We write $E\leq F$ if $\max E\leq \min F$.  We write $n\leq F$ if $n\leq \min F$.  By convention, $\min \varnothing=\infty$, $\max \varnothing =0$.  A family $\mathcal{F}\subset [\mathbb{N}]^{<\omega}$ is called \emph{hereditary} if, whenever $E\in \mathcal{F}$ and $F\subset E$, $F\in \mathcal{F}$.  Last, note that a hereditary family is compact if and only if it contains no strictly ascending chains.  

Given two (finite or infinite) subsequences $(k_n),(\ell_n)\subset \mathbb{N}$ of the same length, we say $(\ell_n)$ is a \emph{spread} of $(k_n)$ if $k_n\leq \ell_n$.  We call a family $\mathcal{F}\subset [\mathcal{N}]^{<\omega}$ \emph{spreading} if it contains all spreads of its elements.  

We next recall the definitions of \emph{the Schreier families}.  Let $$S_0=\bigl\{\{n\}:n\in \mathbb{N}\bigr\}\cup \{\varnothing\}.$$

Next, let $\alpha<\omega_1$ and assume that $S_\beta$ has been defined for all ordinals $\beta\leq \alpha$.  Let $$S_{\alpha+1}=\Bigl\{\underset{n=1}{\overset{m}{\bigcup}}E_n: m\leq E_1<\ldots<E_m, E_n\in S_\alpha \text{\ for all\ }n\Bigr\}.$$

If $\alpha<\omega_1$ is a limit ordinal, take $\alpha_n$ so that $\alpha_n\uparrow \alpha$.  Define $$S_\alpha=\{F:\exists n\leq F\in S_{\alpha_n}\}.$$

The Schreier families, thus defined, are compact, hereditary, and spreading.  Note that for $\alpha$ a limit ordinal, $S_\alpha$ depends upon the choice of the sequence $(\alpha_n)$.  This will not affect the properties of $S_\alpha$ relied upon in this paper.

Recall that $c_{00}$ denotes all finitely nonzero sequences in $\mathbb{R}$.  For $x=(x_n)\in c_{00}$ and $E\in [\mathbb{N}]^{<\omega}$, we define $Ex=(\chi_E(n)x_n)$, the projection of $x$ onto $E$.  For a countable ordinal $\alpha$, define the norm $\|\cdot\|_{\alpha}$ on $c_{00}$ by $$\|x\|_\alpha=\underset{E\in S_\alpha}{\max}\|Ex\|_1.$$

Here, $\|\cdot\|_1$ denotes the $\ell_1$ norm.  We let $X_\alpha$ be the completion of $c_{00}$ under the norm $\|\cdot\|_\alpha$, and call this space the \emph{Schreier space of order }$\alpha$.  We note that the canonical basis $(e_n)$ for $c_{00}$ becomes a normalized, $1$-unconditional basis for $X_\alpha$.  Moreover, since the Schreier families are spreading, $X_\alpha$ is $1$-right dominant for each $\alpha<\omega_1$.

\begin{proposition}
The basis $(e_n)$ of $X_\alpha$ satisfies subsequential $2$-$X_\alpha$-upper block estimates in itself.  \end{proposition}

\begin{proof}

First, choose sequences $(m_n),(k_n)$ with $m_n\leq k_n<m_{n+1}$.  We prove that $(e_{m_n})$ $2$ dominates $(e_{k_n})$.  Fix $(a_n)\in c_{00}$.  Let $x=\displaystyle\sum_{n=1}^\infty a_ne_{m_n}$, $y=\displaystyle\sum_{n=1}^\infty a_ne_{k_n}$.  Fix $E\in S_\alpha$ so that $\|Ey\|_1=\|y\|_\alpha$.  Let $N=\{n: k_n\in E\}$ and $K=\{k_n:n\in N\}$, $M=\{m_n: n\in N\}$.  Then $K\subset E$, and $K\in S_\alpha$.  

We note that $\|Ey\|_1=\displaystyle\sum_{n\in N} |a_n|$.  

If $|a_{\min N}|>\frac{1}{2} \displaystyle\sum_{n\in N}|a_n|=\frac{1}{2}\|y\|_\alpha$, let $F=\{m_{\min n}\}$.  Otherwise, let $$F=\{m_n: n\in N, n\neq \min N\}.$$  In the second case, $F$ is a spread of $\{k_n: n\in N, n\neq \max N\}\subset K$.  We note that the first case failing means $N$ must have at least two elements, so this set is nonempty.  In either case, $F\in S_\alpha$.  

In the first case, $\|Fx\|_1=|a_{\min N}|>\frac{1}{2}\|y\|_\alpha$.  In the second case, $$\|Fx\|_1=\displaystyle\sum_{\underset{n> \text{\ }\min N}{n\in N}}|a_n|\geq \frac{1}{2}\displaystyle\sum_{n\in N}|a_n|=\frac{1}{2}\|y\|_\alpha.$$

Thus $\|x\|_\alpha \geq \|y\|_\alpha$, and we have that $(e_{m_n})$ $2$ dominates $(e_{k_n})$.  

Next, fix any normalized block sequence $(x_n)\subset X_\alpha$.  Let $\min \text{supp\ } x_n=m_n$.  Fix $(a_n)\in c_{00}$, and choose $E\in S_\alpha$ so that $$\Bigl\|E\displaystyle\sum_{n=1}^\infty a_nx_n\Bigr\|_1=\Bigl\|\displaystyle\sum_{n=1}^\infty a_nx_n\Bigr\|_\alpha.$$

Let $N=\{n: E\cap \text{supp\ } x_n\neq \emptyset\}$.  For each $n\in N$, choose $k_n\in E\cap \text{supp}x_n$.  Then we have \begin{align*} \Bigl\|\displaystyle\sum_{n=1}^\infty a_nx_n\Bigr\|_\alpha & = \Bigl\|E\displaystyle\sum_{n\in N}a_nx_n\Bigr\|_1=\displaystyle\sum_{n\in N}|a_n|\|Ex_n\|_1\\ & \leq \displaystyle\sum_{n\in N}|a_n|=\Bigl\|E\displaystyle\sum_{n\in N}a_n e_{k_n}\Bigr\|_1\leq \Bigl\|\displaystyle\sum_{n\in N}a_ne_{k_n}\Bigr\|_\alpha \\ & \leq \frac{1}{2}\Bigl\|\displaystyle\sum_{n\in N}a_ne_{m_n}\Bigr\|_\alpha \leq \frac{1}{2}\Bigl\|\displaystyle\sum_{n=1}^\infty a_n e_{m_n}\Bigr\|_\alpha. \end{align*}

\end{proof}

\vspace{5mm}

The space $X_\alpha$ is embeddable in $C([1,\omega^{\omega^\alpha}])$ \cite{AJO}.  Consequently, $X_\alpha$ is $c_0$ saturated for each $\alpha<\omega_1$, and it easily follows that $X_\alpha$ cannot be block-stable for $0<\alpha$.  This means $(e_n)$ cannot satisfy subsequential $X_\alpha$-lower block estimates in $X_\alpha$.  For our purposes, however, one-sided estimates will suffice.    

We conclude this section by recalling a theorem of Gasparis from infinite Ramsey theory.  

\begin{theorem}\cite{G}
If $\mathcal{F}, \mathcal{G}\subset[\mathbb{N}]^{<\omega}$ are hereditary and $N\in [\mathbb{N}]$, then there exists $M\in [N]$ so that either $$\mathcal{F}\cap [M]^{<\omega}\subset \mathcal{G}\text{\ or\ }\mathcal{G}\cap [M]^{<\omega}\subset \mathcal{F}.$$

\end{theorem}

\section{Ordinal Indices}

Let $\sigma$ be an arbitrary set.  We let $\sigma^{<\omega}$ denote all finite sequences in $\sigma$, including the sequence of length zero, denoted $(\varnothing)$.  A \emph{tree} on $\sigma$ is a nonemptyset subset $\mathcal{F}\subset \sigma^{<\omega}$ closed under taking initial segments.  We call a tree \emph{hereditary} if every subsequence of a member of $\mathcal{F}$ is a member of $\mathcal{F}$.  

If ${\bf x}=(x_1,\ldots,x_m)$ and ${\bf y}=(y_1, \ldots, y_n)$, we denote the concatenation of ${\bf x}$ with ${\bf y}$ by $({\bf x},{\bf y})$.  If $\mathcal{F}\subset \sigma^{<\omega}$ and ${\bf x}\in \sigma^{<\omega}$, then $$\mathcal{F}({\bf x})=\{{\bf y}\in \sigma^{<\omega}: ({\bf x},{\bf y})\in \mathcal{F}\}.$$

If $\mathcal{F}$ is a tree on $\sigma$ and $\mathcal{F}({\bf x})\neq \varnothing$, then $\mathcal{F}({\bf x})$ is also a tree on $\sigma$.  If $\mathcal{F}$ is hereditary, so is $\mathcal{F}({\bf x})$ and $\mathcal{F}({\bf x})\subset \mathcal{F}$.     

If $\sigma^\omega$ is the set of all (infinite) sequences in $\sigma$, $S\subset \sigma^\omega$, and $\mathcal{F}$ is a tree on $\sigma$, we define the \emph{S-derivative} $\mathcal{F}'_S$ of $\mathcal{F}$ by $$\mathcal{F}'_S=\{{\bf x}\in \sigma^{<\omega}:\exists (y_i)\in S \text{\ so that\ }({\bf x},y_n)\in \mathcal{F}\text{\ } \forall\text{\ } n\}.$$

We next define higher order derivatives of the tree $\mathcal{F}$.  $$ \mathcal{F}^{(0)}_S=\mathcal{F}$$ $$\mathcal{F}^{(\alpha+1)}_S  =\bigl(\mathcal{F}^{(\alpha)}_S\bigr)'_S\text{\ \ \ for all\ }\alpha<\omega_1$$ $$ \mathcal{F}^{(\alpha)}_S  =\underset{\beta<\alpha}{\bigcap}\mathcal{F}^{(\beta)}_S\text{\ \ \ for a limit ordinal\ }\alpha<\omega_1$$

It is clear that these collections are decreasing with respect to containment as the ordinal increases, and that $\mathcal{F}^{(\alpha)}_S\neq \varnothing$ is a tree whenever it is nonempty.  

We define the \emph{S-index} of $\mathcal{F}$ by $I_S(\mathcal{F})=\min \{\alpha: \mathcal{F}^{(\alpha)}_S= \varnothing\}$ if such an $\alpha<\omega_1$ exists, and $I_S(\mathcal{F})=\omega_1$ otherwise.  

We outline the indices which will be of particular interest to us.  If $\mathcal{F}\subset [\mathbb{N}]^{<\omega}$ is a hereditary family, we can consider it as a hereditary tree on $\mathbb{N}$.  If $S$ is the set of strictly increasing subsequences of $\mathbb{N}$ and $\mathcal{F}$ is compact and hereditary, then $I_S(\mathcal{F})=I_{CB}(\mathcal{F})$, the \emph{Cantor-Bendixson index} of $\mathcal{F}$ as a topological space.  We note that the Cantor-Bendixson index is a topological invariant.  Moreover, if $\mathcal{F}, \mathcal{G}\subset [\mathbb{N}]^{<\omega}$ are compact, hereditary, and $\mathcal{F}\subset \mathcal{G}$, $I_{CB}(\mathcal{F})\leq I_{CB}(\mathcal{G})$.  

If $\sigma$ is any set and $S=\sigma^{\omega}$, then the index $I_S(\mathcal{F})$ is called the \emph{order} of the tree $\mathcal{F}$, denoted $o(\mathcal{F})$.  We note that, since $S$ is as large as possible, the order is the largest possible ordinal index.  That is, if $S'\subset \sigma^\omega$ and $\mathcal{F}$ is a tree on $\sigma$, $I_{S'}(\mathcal{F})\leq o(\mathcal{F})$.  

Next, we consider the case of a Banach space $X$ and $S$ the collection of all weakly null sequences in the unit sphere $S_X$.  In this case, for a tree $\mathcal{F}$ on $S_X$, we denote this index, called \emph{the weak index}, by $I_S(\mathcal{F})=I_w(\mathcal{F})$.  

Our last example is the \emph{block index}.  If $Z$ is a Banach space with an FDD $E=(E_n)$, a \emph{block tree of} $(E_n)$ \emph{in} $Z$ is a tree $\mathcal{F}$ on $S_Z$ so that each element is a (finite) block sequence of $(E_n)$.  We let $S$ be the set of infinite normalized block sequences of $(E_n)$ in $Z$.  In this case, the $S$-index of a block tree $\mathcal{F}$, denoted $I_{\text{bl}}(\mathcal{F})$, is the block index of $\mathcal{F}$.  We note that $(E_n)$ is shrinking if and only if every normalized block sequence is weakly null.  This means that for any block tree $\mathcal{F}$ in $S_Z$, $I_{\text{bl}}(\mathcal{F})\leq I_w(\mathcal{F})$.  In all cases with which we are concerned, the block index will be with respect to a specified FDD or some blocking thereof.  Since the block index of a tree with respect to one FDD is the same as that of the same tree with respect to any blocking of that FDD, there will be no ambiguity.  

A set $S\subset \sigma^{\omega}$ \emph{contains diagonals} if every subsequence of a sequence in $S$ also lies in $S$ and for every sequence $({\bf x}_n)\subset S$, there exist $i_1<i_2<\ldots$ in $\mathbb{N}$ so that $(x_{n,i_n})\in S$, where ${\bf x}_n=(x_{n,i})_i$.  The sets $S$ above used to give the Cantor-Bendixson index, the order, and the block index of a tree all clearly contain diagonals.  If $X^*$ is a separable Banach space, then the weak topology on $B_X$ is metrizable and the set of weakly null sequences in $S_X$ contains diagonals.

Given a tree $\mathcal{F}\subset [\mathbb{N}]^{<\omega}$ on $\mathbb{N}$, a family $(x_F)_{F\in \mathcal{F}\setminus\{\varnothing\}}$ in $\sigma$ will be considered as the tree $$\Bigl\{\bigl(x_{\{m_1\}}, x_{\{m_1, m_2\}},\ldots, x_{\{m_1, \ldots, m_k\}}\bigr):k\geq 0, \{m_1, \ldots, m_k\}\in \mathcal{F}\Bigr\}$$

on $\sigma$.  

With this convention, we can state a special case of a proposition which has been very useful in computing certain ordinal indices. 

\begin{proposition}\cite[Proposition 5]{OSZ2}
Let $\sigma$ be an arbitrary set and let $S\subset \sigma^{<\omega}$.  If $S$ contains diagonals, then for a tree $\mathcal{F}$ on $\sigma$ and for a countable ordinal $\alpha$, the following are equivalent.  \begin{enumerate} \item $\omega^\alpha<I_S(\mathcal{F})$ \item There is a family $(x_F)_{F\in S_\alpha\setminus\{\varnothing\}}\subset \mathcal{F}$ so that for all $F\in S_\alpha\setminus\text{MAX}(S_\alpha)$, the sequence $(x_{F\cup\{n\}})_{n>\max F}$ is in $S$. \end{enumerate} \end{proposition}

We need a few more pieces of notation to relay some useful propositions.  If $X$ is a separable Banach space, $\mathcal{F}\subset S_X^{<\omega}$, and $\overline{\varepsilon}=(\varepsilon_n)\subset (0,1)$, we write $$\mathcal{F}^X_{\overline{\varepsilon}}=\Bigl\{(x_n)_{n=1}^N\in S_X^{<\omega}: N\in \mathbb{N}, \exists (y_n)_{n=1}^N\in \mathcal{F}, \|x_n-y_n\|\leq \varepsilon_n,\forall n=1, \ldots, N\Bigr\}.$$

Let $Z$ be a Banach space with FDD $E=(E_n)$, and let $\mathcal{F}$ be a block tree of $(E_n)$ in $Z$.  We write $\Sigma(E,Z)$ for the set of all finite, normalized block sequences of $(E_n)$ in $Z$.  For $\overline{\varepsilon}=(\varepsilon_n)\subset (0,1)$, we let $$\mathcal{F}^{E,Z}_{\overline{\varepsilon}}=\mathcal{F}^Z_{\overline{\varepsilon}}\cap \Sigma(E,Z).$$

Last, the \emph{compression} $\mathcal{\tilde{F}}$ of $\mathcal{F}$ is $$\mathcal{\tilde{F}}=\Bigl\{F\in [\mathbb{N}]^{<\omega}:\exists (z_n)_{n=1}^{|F|}\in \mathcal{F}, F=\{\min \text{supp}_E z_n:n=1, \ldots, |F|\}\Bigr\}.$$

\begin{proposition}\cite[Proposition 6]{OSZ2} Let $X\subset Y$ be Banach spaces with separable duals, and let $\mathcal{F}\subset S_X^{<\omega}$ be a tree on $S_X$.  Then for all $\overline{\varepsilon}=(\varepsilon_n)\subset (0,1)$ we have $I_w(\mathcal{F}_{\overline{\varepsilon}}^Y)\leq I_w(\mathcal{F}^X_{5\overline{\varepsilon}})$. \end{proposition}

\begin{proposition}\cite[Proposition 8]{OSZ2}
Let $Z$ be a Banach space with FDD $E=(E_n)$.  Let $\mathcal{F}$ be a hereditary block tree of $(E_n)$ in $Z$.  Then for all $\overline{\varepsilon}=(\varepsilon_n)\subset (0,1)$ and for all limit ordinals $\alpha$, if $I_{\text{bl}}(\mathcal{F}_{\overline{\varepsilon}}^{E,Z})<\alpha$, then $I_{CB}(\mathcal{\tilde{F}})<\alpha$.  \end{proposition}

Next, we have the Bourgain $\ell_1$ index of a Banach space.  For a Banach space $X$ and $K\geq 1$, we define $$T(X,K)=\Bigl\{(x_n)\in S^{<\omega}_X: (x_n) \text{\ is\ }K\text{\ basic\ }, K\Big\|\displaystyle\sum a_nx_n\Bigr\|\geq \displaystyle\sum |a_n| \text{\ \ }\forall(a_n)\subset \mathbb{R}\Bigr\}.$$

Similarly, if $X$ has basis $(e_n)$, we define $T_b(X,K,(e_n))=T(X,K)\cap \Sigma((e_n),X)$.  These are hereditary trees.  We define $I(X,K)=o(T(X,K))$, the order of the tree $T(X,K)$.  We define $I_b(X,K,(e_n))=o(T_b(X,K,(e_n))$.  Finally, we define $I(X)=\underset{K\geq1}{\sup}I(X,K)$, and $I_b(X,(e_n))=\underset{K\geq 1}{\sup}I_b(X,K,(e_n))$.  Roughly speaking, the index $I$ gives some measure of the complexity of the finite dimensional $\ell_1$ structures contained within $X$.  The $I_b$ index gives some measure of the complexity of the finite dimensional $\ell_1$ structures contained within the block basic sequences of $X$.  It is important to note that, in general, $I_b$ is distinct from the previously defined block index $I_{bl}$.  Moreover, by \cite[Theorem 3.14]{AJO}, $I_b(X,(e_n))=\omega_1$ if and only if $X$ contains an isomorphic copy of $\ell_1$.    

\vspace{5mm}  

Last, we recall the Szlenk index of a separable Banach space.  Let $X$ be a separable Banach space, and $K$ a weak$^*$ compact subset of $X^*$.  For $\varepsilon>0$, we define $$(K)'_\varepsilon=\Bigl\{z\in K: \text{For all\ }w^*\text{ -neighborhoods\ }U \text{\ of\ }z, \text{diam}(U\cap K)>\varepsilon\Bigr\}.$$  It is easily verified that $(K)'_\varepsilon$ is also weak$^*$ compact.  We let $$P_0(K, \varepsilon)=K$$

$$P_{\alpha+1}(K, \varepsilon)=(P_\alpha(K, \varepsilon))'_\varepsilon \text{\ \ \ }\alpha<\omega_1$$

$$P_\alpha(K, \varepsilon)=\underset{\beta<\alpha}{\bigcap}P_\beta(K, \varepsilon) \text{\ \ \ }\alpha<\omega_1,\text{\ \ }\alpha  \text{\ a limit ordinal.}$$

If there exists some $\alpha<\omega_1$ so that $P_\alpha(K, \varepsilon)=\varnothing$, we define $$\eta(K, \varepsilon)=\min \{\alpha: P_\alpha(K)=\varnothing\}.$$  Otherwise, we set $\eta(K, \varepsilon)=\omega_1$.  Then we define the Szlenk index of a Banach space $X$, denoted $Sz(X)$, to be $$Sz(X)=\underset{\varepsilon>0}{\sup}\text{\ }\eta(B_{X^*},\varepsilon).$$

The Szlenk index is one of several slicing indices.  The following two facts come from \cite{SZLENK}. \begin{enumerate}\item  For a Banach space $X$, $Sz(X)<\omega_1$ if and only if $X^*$ is separable,  \item If $X\leq Y$, $Sz(X)\leq Sz(Y)$. \end{enumerate}

The above definition of the index is, in some cases, intractable.  A connection between weak indices and the Szlenk index has been very useful in computations.  For this, we will be concerned with a specific type of tree.  

For a Banach space $X$ and $\rho\in (0,1]$, we let $$\mathcal{H}^X_\rho=\Bigl\{(x_n)\in S_X^{<\omega}: \Bigl\|\displaystyle\sum a_nx_n\Bigr\|\geq \rho \displaystyle\sum a_n \text{\ \ \ }\forall (a_n)\subset \mathbb{R}^+\Bigr\}.$$

Clearly $\mathcal{H}^X_\rho$ is a hereditary tree on $S_X$ for all $\rho\in (0,1]$.  We collect two tools which will facilitate the computation of the Szlenk indices of the Schreier spaces.  

\begin{theorem}\cite[Theorems 3.22, 4.2]{AJO}
If $X$ is a Banach space with $X^*$ separable, there exists some ordinal $\beta<\omega_1$ so that $Sz(X)=\omega^\beta$. 
Moreover, $$Sz(X)=\underset{\rho\in (0,1)}{\sup}I_w(\mathcal{H}^X_\rho).$$
\end{theorem}

\begin{corollary} Let $V$ be a Banach space with normalized, $1$-unconditional, shrinking basis $(v_n)$.  If $Z$ is a Banach space with shrinking FDD $E$ which satisfies subsequential $V$-upper block estimates, then $Sz(Z)\leq Sz(V)$.  \end{corollary} 

\begin{proof} The proof is a generalization of Proposition $17$ of \cite{OSZ2}.  

Let $\alpha<\omega_1$ be such that $Sz(V)=\omega^\alpha$.  Assume $Sz(Z)>\omega^\alpha$.  By Proposition $4.4$ there exists some $\rho\in (0,1]$ so that $I_w(\mathcal{H}^Z_\rho)>\omega^\alpha$.  Then by Proposition $4.1$ there exists some normalized tree $(x_E)_{E\in S_{\alpha}\setminus \{\varnothing\}}\subset \mathcal{H}^Z_\rho$ so that for each $E\in S_\alpha\setminus \text{MAX}(S_\alpha)$, $(x_{E\cup \{n\}})_{n>\max E}$ is weakly null.  We can, by shrinking $\rho$ and using standard perturbation and pruning arguments, assume that $(x_E)_{E\in S_\alpha\setminus \{\varnothing\}}$ is a block tree with the added requirement that each branch is a block sequence.  Let $C\geq 1$ be such that $E$ satisfies subsequential $C$-$V$-upper block estimates in $Z$.  Then $(v_{m_E})$ is a normalized block tree, where $m_E=\min \text{supp\ }x_E$.  Because the braches of this tree $C$ dominate the branches of the tree $(x_E)_{E\in S_\alpha\setminus\{\varnothing\}}$, $(v_{m_E})_{ E\in S_\alpha\setminus \{ \varnothing \} }\subset \mathcal{H}^V_{\rho C^{-1}}$.

Moreover, since all nodes are block sequences, $m_{E\cup \{n\}}\to \infty$ as $n\to \infty$ if $E\in S_\alpha \setminus\text{MAX}(S_\alpha)$.  Because the basis for $V$ is shrinking, $(v_{m_{E\cup \{n\}}})_{n>\max E}$ is weakly null for such $E$.  But the existence of such a tree, again by Proposition $4.1$, means $Sz(V)>\omega^\alpha$.  But this is a contradiction, and we have the result.  

\end{proof}

\vspace{5mm}

Last, we make an observation regarding the spaces which will be our main tools.  

\begin{proposition} For $\alpha<\omega_1$, $Sz(X_\alpha)=\omega^{\alpha+1}$. \end{proposition}

\begin{proof} By \cite[Theorem 5.5]{AJO}, we know $I_b(X_\alpha, (e_n))=\omega^{\alpha+1}$.  Then $X_\alpha$ contains no copy of $\ell_1$, or else $I(X_\alpha)=\omega_1$.  Since the canonical basis is unconditional, this means it must be shrinking.  Therefore the tree $(e_{\max E})_{E\in S_\alpha}$ is such that all nodes are weakly null.  Moreover, each branch is of the form $(e_n)_{n\in E}$ for some $E\in S_\alpha$.  This branch is isometrically $\ell_1^{|E|}$, and so this tree is contained in $\mathcal{H}^{X_\alpha}_1$.  By Proposition $4.1$, $I_w(\mathcal{H}_1^{X_\alpha})>\omega^\alpha$.  

Thus we need only show that $Sz(X_\alpha)\leq \omega^{\alpha+1}$.  If not, there must exist some $\rho\in (0,1]$ so that $I_w(\mathcal{H}_\rho^{X_\alpha})>\omega^{\alpha+1}$.  By standard perturbation arguments, we can assume that there is a block tree $(x_E)_{E\in S_{\alpha+1}\setminus \{\varnothing\}}\subset H^{X_\alpha}_\rho$.  This means that if $(x_n)$ is a branch in the tree and $(a_n)\subset \mathbb{R}^+$, $$\Bigl\|\displaystyle\sum a_nx_n\Bigr\|\geq \rho \displaystyle\sum a_n.$$  

But because $(x_n)$ is a block sequence and the canonical basis for the Schreier space $X_\alpha$ is $1$-unconditional, this means $\rho^{-1}\Bigl\|\displaystyle\sum a_nx_n\Bigr\|\geq \displaystyle\sum |a_n|$ for all $(a_n)\subset \mathbb{R}$.  Thus $(x_E)_{E\in S_{\alpha+1}\setminus \{\varnothing\}}\subset T_b(X_\alpha,\rho^{-1}, (e_n))$.  Then \begin{align*} \omega^{\alpha+1}&=I_b(X_\alpha, (e_n))\geq I_b(X_\alpha, \rho^{-1}, (e_n)) \geq o(S_{\alpha+1})>\omega^{\alpha+1}.\end{align*}

This is a contradiction, and we get the result.  

\end{proof}

\section{proof of main theorems}

We include the proof of the first theorem here for completeness.  It can be found in \cite{FOSZ}, where slightly stronger hypotheses were used.  

\begin{theorem}\cite[Theorem 1.1]{FOSZ}
If $V$ is a Banach space with normalized, $1$-unconditional, shrinking, right dominant basis $(v_n)$ which satisfies subsequential $V$-upper block estimates in $V$, and $X$ is a Banach space with separable dual, then the following are equivalent.
\begin{enumerate} \item $X$ satisfies subsequential $V$-upper tree estimates,
\item $X$ is a quotient of a space $Z$ with $Z^*$ separable and $Z$ has subsequential $V$-upper tree estimates,
\item $X$ is a quotient of a space $Z$ with a shrinking FDD satisfying subsequential $V$-upper block estimates,
\item There exists a $w^*-w^*$ continuous embedding of $X^*$ into $Z^*$, a space with boundedly-complete FDD $(F_i^*)$ satisfying subsequential $V^*$-lower block estimates,
\item $X$ is isomorphic to a subspace of a space with a shrinking FDD satisfying subsequential $V$-upper block estimates.  \end{enumerate}

\end{theorem}

\begin{proof}

First, we note that $(5)\Rightarrow (1)$ and $(3)\Rightarrow (2)$ are trivial.  

$(1)\Rightarrow (4)$ Let $D\geq 1$ be such that $(v_n)$ is $D$-right dominant.  By the remark preceding Proposition $1$ of \cite{OSZ1}, $(v_n^*)$ is $D$-left dominant.  By \cite[Corollary 8]{DFJP}, there exists a space $Z$ with shrinking, bimonotone FDD $E=(E_n)$ for which there is a quotient map $Q:Z\to X$.  The map $Q^*:X^*\to Z^*$ is an into isomorphism.  After renorming $X$ if necessary, we can assume that $X$ has the quotient norm induced by $Q$, and so $Q^*$ is an isometric embedding.  By Proposition $2.8$, $X^*$ satisfies subsequential $C$-$V^*$-lower $w^*$ tree estimates for some $C\geq 1$.  As $Q^*X^*\subset Z^*$ is $w^*$ closed, we may apply Proposition $2.15$ with $$A=\bigl\{(i_n,x_n)_{n=1}^\infty \in [\mathbb{N}\times S_{Q^*X^*}]: (x_n)\text{\ } C-\text{dominates\ }(v_{i_n})\bigr\}$$

and $\varepsilon>0$ so that $$\overline{A}_\varepsilon \subset \bigl\{(i_n,x_n)\in [\mathbb{N}\times S_{Q^*X^*}]: (x_n) \text{\ }2CD-\text{\ dominates\ } (v_{i_n})\bigr\}.$$

This gives sequences $(K_n)\in [\mathbb{N}]$ and $\overline{\delta}=(\delta_n)\subset (0,1)$ and a blocking $(F_n)$ of $(E^*_n)$ such that if $(x_n)\subset S_{Q^*X^*}$ and $\|x_n-P^F_{(r_{n-1},r_n)}x_n\|<2\delta_n$ for some sequence $(r_n)\in [\mathbb{N}]$, then $(K_{r_{n-1}}, x_n)\in \overline{A}_\varepsilon$.  Hence, the sequence $(x_n)$ $2CD$-dominates $(v_{K_{r_{n-1}}})$.  

Take a blocking $G=(G_n)$ of $(F_n)$ defined by $G_n=\underset{j=m_{n-1}+1}{\overset{m_n}{\bigoplus}}F_j$ for some $(m_n)\in [\mathbb{N}]$ such that there exists $(e_n)\subset S_{Q^*X^*}$ with $\|e_n -P^G_n e_n\|<\frac{\delta_n}{2}$ for all $n$.  In order to continue, we need the following result from \cite{OS1} which is based on an argument due to W.B. Johnson in \cite{J}.  \cite[Corollary 4.4]{OS1} was stated for reflexive spaces.  Here we state it for $w^*$-closed subspaces of dual spaces with a boundedly-complete FDD:  The proof is easily seen to work in this case.  Also note that conditions $(4)$ and $(5)$ which were not stated in \cite{OS1} follow easily from the proof.  

\begin{proposition}\cite[Lemma 4.3, Corollary 4.4]{OS1}

Let $Y$ be a $w^*$-closed subspace of a dual space $Z$ with boundedly-complete FDD $A=(A_n)$ having projection constant $K$.  Let $\overline{\eta}=(\eta_n)\subset (0,1)$ with $\eta_n\downarrow 0$.  Then there exists $(N_n)_{n=1}^\infty\in [\mathbb{N}]$ such that the following holds.  Given $(k_n)_{n=0}^\infty \in [\mathbb{N}]$ and $x\in S_Y$, there exists $x_n\in Y$ and $t_n\in (N_{k_{n-1}}, N_{k_n})$ for all $n\in \mathbb{N}$ with $N_0=0$ and $t_0=0$ such that 

\begin{enumerate}
\item $x=\displaystyle\sum_{n=1}^\infty x_n$ and for all $n\in \mathbb{N}$ we have,
\item either $\|x_n\|<\eta_n$ or $\|x_n-P^A_{(t_{n-1},t_n)}x_n\|<\eta_n \|x_n\|$,
\item $\|x_n-P^A_{(t_{n-1},t_n)}x\|<\eta_n$,
\item $\|x_n\|<K+1$,
\item $\|P^A_{t_n}x\|<\eta_n$.

\end{enumerate}\end{proposition}

We apply Proposition $5.2$ to $Y=Q^*X^*$, $A=G$, and $\overline{\eta}=\overline{\delta}$ which gives a sequence $(N_n)$.  We set $H_n=\underset{i=N_{n-1}+1}{\overset{N_n}{\bigoplus}} G_i$, for each $n\in \mathbb{N}$.  To make notation easier we let $V^*_M=(v^*_{M_n})$ be the subsequence of $(v^*_n)$ defined by $M_n=K_{m_{N_n}}$.  

Fix $x\in S_{Q^*X^*}$ and a sequence $(k_n)_{n=0}^\infty\in [\mathbb{N}]$.  The proof of \cite[Theorem 4.1(a)]{OSZ1} shows $$\Bigl\|\displaystyle\sum_{n=1}^\infty \|P^H_{[k_{n-1},k_n)}x\|_{Z^*} v^*_{M_{k_{n-1}}}\Bigr\|_{V^*}\leq 4D^2C(1+2\Delta+2)+2+3\Delta.$$

where $\Delta=\displaystyle\sum_{n=1}^\infty \delta_n$.  Thus the norms $\|\cdot\|_{Z^*}$ and $\|\cdot\|_{(Z^*)^{V^*_M}}$ are equivalent on $Q^*X^*$.  As the norm on each $H_n$ is unchanged, a coordinate-wise null sequence in $Q^*X^*\subset Z$ will still be coordinate-wise null in $(Z^*)^{V^*_M}$.  Hence the map $Q^*: X^*\to (Z^*)^{V^*_M}$ is still $w^*-w^*$ continuous.  

We have that $(Z^*)^{V^*_M}$ has a boundedly-complete FDD $(H_n)$ which satisfies subsequential $V^*_M$-lower block estimates by Propositions $2.4$ and Lemma $2.11$.  We can now fill in the FDD as in Lemma $2.13$ to get $W=(Z^*)^{V^*_M}\oplus_\infty V^*_{\mathbb{N}\setminus M}$ with FDD $(F_n)$.  The natural embedding of $(Z^*)^{V^*_M}$ into $W$ is $w^*-w^*$ continuous.  Hence there is a $w^*-w^*$ continuous embedding of $X^*$ into $W$.  Finally, from the fact that $(H_n)$ satisfies subsequential $V^*_M$-lower block estimates in $(Z^*)^{V^*_M}$, we get that $(F_n)$ satisfies subsequential $V^*$-lower block estimates in $W$.

$(4)\Rightarrow(3)$ This is clear because if $(F^*_n)$ is a boundedly-complete FDD of $Z^*$, then $(F_n)$ is a shrinking FDD of its predual $Z$ and a $w^*-w^*$ continuous embedding $T:X^*\to Z^*$ must be the adjoint of some quotient map $Q:Z\to X$.  Also, $(F^*_n)$ having subsequential $V^*$-lower block estimates is equivalent to $(F_n)$ having subsequential $V$-upper block estimates by Prposition $2.2$.  

$(3)\Rightarrow (1)$ Let $(F_n)$ be a bimonotone, shrinking FDD which satisfies subsequential $B$-$V$-upper block estimates in $Z$, and $Q:Z\to X$ a quotient map.  Let $D$ be such that $(v_n)$ is $D$-right dominant.  There exists $C>0$ such that $B_X\subset Q(CB_Z)$.  We will need a lemma from \cite{FOSZ}.

\begin{lemma}\cite[Lemma 3.2]{FOSZ} Let $X$ and $Z$ be Banach spaces, $F=(F_n)$ a bimonotone FDD for $Z$, and $Q:Z\to X$ a quotient map.  If $(x_n)\subset S_X$ is weakly null and $Q(CB_Z)\supset B_X$ for some $C>0$, then for all $\varepsilon>0$ and $n\in \mathbb{N}$, there exists $N\in \mathbb{N}$ and $z\in 2CB_Z$ such that $P_{[1,n]}z=0$ and $\|Qz-x_N\|<\varepsilon$.  \end{lemma}

Let $(x_t)_{t\in T^\text{even}_\infty}\subset S_X$ be a weakly null even tree in $X$, and let $\eta\in (0,1)$.  By Lemma $5.3$ we may pass to a full subtree $(x'_t)_{t\in T^\text{even}_\infty}$ so that there exists a block tree $(z_t)_{t\in T^\text{even}_\infty}\subset 2CB_Z$ so that $\|Q(z_t)-x'_t\|<\eta2^{-\ell}$ for all $\ell\in \mathbb{N}$ and $t=(k_1, \ldots, k_{2\ell})\in T^\text{even}_\infty$.  Now choose $1=k_1<k_2<\ldots$ such that $\max \text{supp} z_{(k_1, \ldots, k_{2n})}<k_{2n+1}<\min \text{supp} z_{(k_1, \ldots, k_{2n+2})}$ for all $n$.  Then $(z_{(k_1, \ldots, k_{2n})})$ is $2BC$-dominated by $(v_{k_{2n-1}})$, and hence $(x'_{(k_1, \ldots, k_{2n})})$ is $3BC$-dominated by $(v_{k_{2n-1}})$ provided $\eta$ was chosen sufficiently small.  Finally, the branch $(k_{2n-1}, x'_{(k_1, \ldots, k_{2n})})$ corresponds to a branch $(\ell_{2n-1}, x_{(\ell_1, \ldots, \ell_{2n})})$ in the original tree with $k_n\leq \ell_n$ for all $n$.  Since $(v_n)$ is right dominant, it follows that $(x_{(\ell_1, \ldots, \ell_{2n})})$ is $3BCD$-dominated by $(v_{\ell_{2n-1}})$.  Thus $X$ satisfies subsequential $3BCD$-$V$-upper tree estimates.

$(2)\Rightarrow (1)$ We assume that $X$ is a quotient of a space $Z$ with separable dual such that $Z$ satisfies subsequential $V$-upper tree estimates.  By $(1)\Rightarrow (3)$ applied to $Z$, $Z$ is a quotient of a space $Y$ with shrinking FDD satisfying subsequential $V$-upper block estimates.  $X$ is then also a quotient of $Y$, so by $(3)\Rightarrow (1)$ we have that $X$ satisfies subsequential $V$-upper tree estimates.

$(1)\Rightarrow (5)$ Our proof will be based on the proof of \cite[Theorem 4.1(b)]{OSZ1}.  Assume $X$ satisfies $V$-upper tree estimates.  By a theorem from Zippin \cite{Z}, we may assume, after renorming $X$ if necessary, that there is a Banach space $Z$ with a shrinking, bimonotone FDD $(F_n)$ and an isometric embedding $i:X\to Z$.  Also, by \cite[Corollary 8]{DFJP} there is a Banach space $W$ with shrinking FDD $(E_n)$ and a quotient map $Q:W\to X$.  Thus we have a quotient map $i^*:Z^*\to X^*$ and an embedding $Q^*:X^*\to W^*$.  We can assume, after renorming $W$ if necessary, that $Q^*$ is an isometric embedding.  Note that $(F_n^*)$, $(E^*_n)$ are boundedly-complete FDDs of $Z^*$ and $W^*$, respectively, and that $X^*$ has the quotient norm induced by $i^*$.  Let $K$ be the projection constant of $(E_n)$ in $W$.  

By Proposition $2.8$, $X^*$ satisfies subsequential $C$-$V^*$-lower $w^*$ tree estimates for some $C\geq 1$.  Choose $D\geq 1$ so that $(v_n)$ is $D$ right dominant.  Since $Q^*X^*$ is $w^*$ closed in $W^*$, we can apply Proposition $2.15$ as in $(1)\Rightarrow (4)$.  That is, after blocking $(E^*_n)$, we can find sequences $(K_n)\in [\mathbb{N}]$, and $\overline{\delta}=(\delta_n)\subset (0,1)$ with $\delta_n\downarrow 0$ such that if $(x_n^*)\subset S_{Q^*X^*}$ is a $2K\overline{\delta}$-skipped block of $(E^*_n)$ with $\|x_n^*-P^{E^*}_{(r_{n-1},r_n)} x_n^*\|<2K\delta_n$ for all $n$, where $1\leq r_0<r_1<\ldots$, then $(v^*_{K_{r_{n-1}}})$ is $2CD$-dominated by $(x_n^*)$ and, moreover, using standard perturbation arguments and making $\overline{\delta}$ smaller if necessary, we can assume that if $(w_n^*)\subset W^*$ satisfies $\|x_n^*-w_n^*\|<\delta_n$ for all $n$, then $(w_n^*)$ is a basic sequence equivalent to $(x_n^*)$ with projection constant at most $2K$.  We can also assume $\Delta=\displaystyle\sum_{n=1}^\infty \delta_n<7^{-1}$.  

Choose a sequence $(\varepsilon_n)\subset (0,1)$ with $\varepsilon_n\downarrow 0$ and $3K(K+1)\displaystyle\sum_{i=n}^\infty \varepsilon_i<\delta^2_n$ for all $n$.  After blocking $(E^*_n)$ if necessary, we can assume that for any subsequent blocking $(D_n)$ of $E^*$ there is a sequence $(e_n)$ in $S_{Q^*X^*}$ such that $\|e_n-P^D_n e_n\|<\frac{\varepsilon_n}{2K}$ for all $n$.  

Using Johnson and Zippin's blocking lemma \cite{JZ} we may assume, after further blocking our FDDs $(E_n^*)$ and $(F^*_n)$ if necessary, that given $k<\ell$, if $z^*\in \bigoplus_{n\in (k,\ell)}F^*_n$ with $\|z^*\|\leq 1$, then $\|P^{E^*}_{[1,k)}Q^*i^*z^*\|<\varepsilon_k$ and $\|P^E_{[\ell,\infty)}Q^*i^*z^*\|<\varepsilon_\ell$, and that this holds if one passes to any further blocking of $(F^*_n)$ and the corresponding blocking of $(E^*_n)$.  Note that the hypotheses of the Johnson-Zippin lemma are not satisfied here, but the proof is seen to apply since we have boundedly-complete FDDs and the map $Q^*i^*$ is $w^*-w^*$ continuous.  

We now continue as in the proof of \cite[Theorem 4.1(b)]{OSZ1}.  We replace $F^*_n$ by the quotient space $\tilde{F}_n=i^*(F^*_n)$.  We let $\tilde{Z}$ be the completion of $c_{00}\Bigl(\oplus \tilde{F}_n\Bigr)$ with respect to the norm $|||\cdot|||$ defined in \cite{OSZ1} on $c_{00}\Bigl(\oplus_{n=1}^\infty \tilde{F}_n\Bigr)$ to be $$|||\tilde{z}|||= \underset{k<m}{\max}\Bigl\|\sum_{n=k}^m i^*(z_n)\Bigr\|,$$ where $\tilde{z} = \sum \tilde{z}_n$.  We obtain a quotient map $\tilde{i}:\tilde{Z}\to X^*$.  We note that the result corresponding to \cite[Proposition 4.9(b),(c)]{OSZ1} are valid here as their proof does not require reflexitivity (part $(a)$ is not required, nor valid, here).  

Finally, we find a blocking $(\tilde{G}_n)$ of $(\tilde{F}_n)$ and a subsequence $V_M^*=(v^*_{m_n})$ such that $\tilde{i}$ is still a quotient map of $\tilde{Z}^{V^*_M}(\tilde{G})$ onto $X^*$ and it is still $w^*-w^*$ continuous (note that $(\tilde{G}_n)$ is boundedly-complete in $\tilde{Z}^{V^*_M}(\tilde{G})$ by Proposition $2.10$).  To find suitable $\tilde{G}$ and $(m_n)$ we follow the proof of \cite[Theorem 4.1(b)]{OSZ1} verbatim.  We only need to note that \cite[Lemma 4.10]{OSZ1} is valid since we are working with boundedly-complete FDDs and $w^*-w^*$ continuous maps.  Note that $\tilde{G}$ satisfies subsequential $V_m^*$-lower block estimates in $\tilde{Z}^{V^*_M}(\tilde{G})$ by Lemma $2.11$.  Again, we fill out the FDD as in Lemma $2.13$ to obtain $Y=\tilde{Z}^{V^*_M}(\tilde{G})\oplus_\infty V^*_{\mathbb{N}\setminus M}$ with FDD satisfying subsequential $V^*$-lower block estimates in $W$.  Since the corresponding FDDs in the sum are boundedly-complete, so is the FDD for $Y$.  The quotient map we have obtained onto $Y$ is therefore the adjoint of an embedding of $X$ into the predual of $Y$.  By Proposition $2.7$, since $Y$ satisfies subsequential $V^*$-lower block estimates, the predual satisfies subsequential $V$-upper block estimates.

\end{proof}

The following proof is similar to that contained in \cite{FOSZ} of a similar statement with the hypothesis of block stability.  For completeness, we include a proof of the more general statement with weaker hypotheses.  In it, we make reference to the class $\mathcal{A}_V$, which was introduced before Proposition $2.7$.

\begin{theorem}\cite[Corollary 3.3]{FOSZ}
Let $V$ be a Banach space with normalized, $1$-unconditional, shrinking, right dominant basis $(v_n)$ which satisfies subsequential $V$-upper block estimates in $V$.  Then the class $\mathcal{A}_V$ contains a universal element $Z$ which has shrinking, bimonotone FDD.   \end{theorem}

\begin{proof} By a result of Schechtman \cite{Sch}, there exists a space $W$ with bimontone FDD $E=(E_n)$ with the property that any space $X$ with bimonotone FDD $F=(F_n)$ naturally almost isometrically embeds into $W$.  Moreover, for any $\varepsilon>0$, there exists a $(1+\varepsilon)$-embedding $T:X\to W$ and $(k_n)\in [\mathbb{N}]$ so that $T(F_n)=E_{k_n}$ and $\displaystyle\sum_{n=1}^\infty P^E_{k_n}$ is a norm-$1$ projection of $W$ onto $T(X)$.  

Since the basis $(v_n^*)$ of $V^*$ is boundedly-complete, it follows from Proposition $2.10$ that the sequence $(E_n^*)$ is a boundedly-complete FDD for the space $(W^{(*)})^{V^*}$.  It follows that $(E_n)$ is a shrinking FDD of the space $Z=\bigl((W^{(*)})^{V^*}\bigr)^{(*)}$ and that $Z^*=(W^{(*)})^{V^*}$.  We denote by $\|\cdot\|_W$, $\|\cdot\|_{W^{(*)}}$, $\|\cdot\|_Z$, and $\|\cdot\|_{Z^*}$ the norms of $W, W^{(*)}, Z,$ and $Z^*$, respectively.  

By Lemma $2.11$, $(E_n)$ satisfies subsequential $V^*$-lower block estimates in $Z^*$. By Proposition $2.7$, $(E_n)$ satisfies subsequential $V$-upper block estimates in $Z$.  This is because $Z^*=Z^{(*)}$.  

If $X$ is any space with separable dual and subsequential $V$-upper tree estimates, $X$ embeds into a space $Y$ with shrinking, bimonotone FDD which satisfies subsequential $V$-upper block estimates.  If we prove the result for $Y$, this will imply the result for $X$, so we can assume that $X$ itself has shrinking, bimonotone FDD $F=(F_n)$ satisfying subsequential $V$-upper block estimates.  By our choice of $W$, we can also assume $X$ is a $1$-complemented subspace of $W$ and that $(F_n)=(E_{k_n})$ for some subsequence $(k_n)$ of $\mathbb{N}$.  It suffices to show that the norms $\|\cdot\|_W$ and $\|\cdot\|_Z$ are equivalent on $X$.  

Let $C\geq 1$ be chosen so that $(E^*_{k_n})$ satisfies subsequential $C$-$V^*$-lower block estimates in $X^*$, $(v_n)$ is $C$-right dominant and satisfies subsequential $C$-$V$-upper block estimates in $V$.  This means $(v_n^*)$ is $C$-left dominant and satisfies subsequential $C$-$V^*$-lower block estimates in $V^*$.  Let $w^*\in c_{00}\bigl(\oplus E^*_{k_n}\bigr)$.  Clearly $\|w^*\|_{W^{(*)}}\leq \|w^*\|_{Z^*}$.  

Choose $1\leq m_0<m_1<\ldots<m_\ell$ so that $$\|w^*\|_{Z^*}=\Bigl\|\displaystyle\sum_{n=1}^\ell \|P^{E^*}_{[m_{n-1},m_n)}w^*\|_{W^{(*)}}v^*_{m_{n-1}}\Bigr\|_{V^*}.$$ 

By discarding terms from the tuple $(m_n)$, we can assume $P^{E^*}_{[m_{n-1},m_n)}w^*\neq 0$ for each $n$.  We must, however, be judicious about choosing how to discard elements from the tuple, since discarding elements from $(m_n)$ affects which of the vectors $v^*_{m_n}$ occur in the sum above.  If $P^{E^*}_{[m_{n-1},m_n)}w^*=0$, we delete $m_{n-1}$, not $m_n$, from the tuple.  This leaves the sum above unchanged.  If we instead delete $m_n$ when $P^{E^*}_{[m_{n-1},m_n)}w^*=0$, this may change the value of the above norm if $(v_n^*)$ fails to be $1$-right dominant.  

Choose $j_1<j_2<\ldots<j_\ell$ so that $k_{j_n}=\min \text{supp}_{E^*} P^{E^*}_{[m_{n-1}, m_n)}w^*$.  Then

\begin{align*} \|w^*\|_{Z^*}&=\Bigl\|\displaystyle\sum_{n=1}^\ell \|P^{E^*}_{[m_{n-1}, m_n)}w^*\|_{W^{(*)}}v^*_{m_{n-1}}\Bigr\|_{V^*} \\ & \leq C\Bigl\|\displaystyle\sum_{n=1}^\ell \|P^{E^*}_{[m_{n-1},m_n)}w^*\|_{W^{(*)}}v^*_{k_{j_n}}\Bigr\|_{V^*} \\ & \leq C^2\Bigl\|\displaystyle\sum_{n=1}^\ell \|P^{F^*}_{[j_n,j_{n+1})}w^*\|_{W^{(*)}}v^*_{j_n}\Bigr\|_{V^*}\leq C^3 \|w^*\|_{W^{(*)}}.\end{align*}

The first inequality comes from the fact that $(v^*_n)$ satisfies subsequential $C$-$V^*$-lower block estimates in $V^*$ and an application of Proposition $2.4$.  The second comes from $C$-left dominance.  The third comes from the fact that $(F^*_n)$ satisfies subsequential $C$-$V^*$-lower block estimates in $X^*$.  

This proves that $\|\cdot\|_{W{(*)}}$ and $\|\cdot\|_{Z^*}$ are equivalent on $c_{00}\bigl(\oplus E^*_{k_n}\bigr)$.  Since $X$ is $1$-complemented in $W$, $X^*$ is $1$-complemented in $W^{(*)}$.  Since $\displaystyle\sum_n P^{E^*}_{k_n}$ is still a norm-$1$ projection from $Z^*$ onto $\overline{c_{00}\bigl(\oplus E_{k_n}\bigr)}^{Z^*}$, it follows that for any $w\in c_{00}\bigl(\oplus E_{k_n}\bigr)$ that $$C^{-3}\|w\|_W\leq \|w\|_Z\leq \|w\|_W,$$
which gives the claim.  

\end{proof}

In the following theorem, $X_\alpha$ denotes the Schreier space of order $\alpha$, defined before Proposition $3.1$, and $(e_i)$ is the unit vector basis of $X_\alpha$.

\begin{theorem}
Let $\alpha<\omega_1$ and $C>2$.  Let $Z$ be a Banach space with a shrinking, bimonotone FDD $(E_n)$, and let $X$ be an infinite dimensional closed subspace.  If $Sz(X)\leq \omega^\alpha$ then there exists $M=(m_n)_{n=0}^\infty\in [\mathbb{N}]$ with $1=m_0\leq m_1<\ldots$ and $\overline{\delta}=(\delta_n)\subset (0,1)$ so that if $(x_n)$ is a normalized $\overline{\delta}$-block sequence of $H=(H_n)$, where $H_n=\underset{i=m_{n-1}}{\overset{m_n-1}{\bigoplus}}E_i$, with $\|x_n-P^H_{[s_{n-1},s_n)}x_n\|<\delta_n$ for some $0\leq s_0<s_1<\ldots$, then $(x_n)$ is $C$-dominated by $(e_{m_{s_{n-1}}})\subset X_\alpha$.  \end{theorem}

\begin{proof}

Fix $2<D<C$.  Choose $\rho\in (0,\frac{1}{3})$ so that $2(1-\rho)^{-2}<D$.  Let $$\mathcal{F}_n=\Bigl\{(x_j)\in S_X^{<\omega}:\Bigl\|\displaystyle\sum a_jx_j\Bigr\|\geq 2\rho^{n+1}\displaystyle\sum a_j \text{\ \ }\forall (a_j)\subset \mathbb{R}^+\Bigr\}.$$

Then $\mathcal{F}_n$ is a hereditary tree on $S^{<\omega}_X$ for each $n$.  Next, for each $n$, fix $\overline{\varepsilon}_n=(\varepsilon_{i,n})_{i=1}^\infty \subset (0,1)$ so that $10\displaystyle\sum_i \varepsilon_{i,n}\leq \rho^{n+1}$ and both functions $i,n\mapsto \varepsilon_{i,n}$ are decreasing.  We note that the requirement that $10\displaystyle\sum_i \varepsilon_{i,n}\leq \rho^n$ means that \begin{equation}(\mathcal{F}_n)_{10\overline{\varepsilon}_n}^Z\subset \Bigl\{(z_j)\in S_Z^{<\omega}:\Bigl\|\displaystyle\sum a_jz_j\Bigr\|\geq \rho^{n+1} \displaystyle\sum a_j\text{\ } \forall (a_j)\subset \mathbb{R}^+\Bigr\}.\label{1}\end{equation}

Let $\mathcal{G}_n=\Sigma(E,Z)\cap (\mathcal{F}_n)_{\overline{\varepsilon}_n}^Z$.  This is a hereditary block tree of $(E_n)$ in $Z$.  Let $\mathcal{\tilde{G}}_n$ be its compression.  By Proposition $4.2$, $I_w((\mathcal{F}_n)_{2\overline{\varepsilon}_n}^Z)\leq I_w((\mathcal{F}_n)_{10\overline{\varepsilon}_n}^X)$. 

Because of the containment in $(1)$, Theorem $4.4$ implies $I_w((\mathcal{F}_n)_{10\overline{\varepsilon}_n}^X)<Sz(X)$.  

Since $(\mathcal{G}_n)^{E,Z}_{\overline{\varepsilon_n}}\subset (\mathcal{F}_n)_{2\overline{\varepsilon}_n}^Z$, we have $I_{\text{bl}}((\mathcal{G}_n)^{E,Z}_{\overline{\varepsilon}_n})\leq I_w((\mathcal{F}_n)_{2\overline{\varepsilon}_n}^Z)$.  Since $Sz(X)$ is a limit ordinal, Proposition $4.3$ gives that  $$I_{CB}(\mathcal{\tilde{G}}_n)<Sz(X)\leq \omega^\alpha.$$

Put $M_0=\mathbb{N}\setminus \{1\}$.  We note that $S_\alpha$ and $\mathcal{\tilde{G}}_1$ are hereditary trees on $[\mathbb{N}]^{<\omega}$.  By Theorem $3.2$, there exists some $M_1\in [M_0\setminus\{\min M_0\}]$ so that either $$S_\alpha\cap [M_1]^{<\omega}\subset \mathcal{\tilde{G}}_1 \text{\ \ or\ \ } \mathcal{\tilde{G}}_1\cap [M_1]^{<\omega}\subset S_\alpha.$$

If we let $M_1=(m_n^{(1)})$, then the map $n\mapsto m_n^{(1)}$ induces a homeomorphism between $S_\alpha$ and $S_\alpha\cap [M_1]^{<\omega}$.  Since $I_{CB}(S_\alpha)=\omega^\alpha+1$, we cannot have the first containment.  Thus $\mathcal{\tilde{G}}_1\cap [M_1]^{<\omega}\subset S_\alpha$.  

Next, assume we have chosen $M_1\supset M_2\supset\ldots M_\ell$ so that $\min M_n<\min M_{n+1}$ for each $1\leq n<\ell$ and $\mathcal{\tilde{G}}_n \cap [M_n]^{<\omega}\subset S_\alpha$ for each $1\leq n\leq \ell$.  Apply Theorem $3.2$ again to get a set $M_{\ell+1}\in [M_\ell\setminus \{\min M_\ell\}]$ so that either $$S_\alpha\cap [M_{\ell+1}]^{<\omega}\subset \mathcal{\tilde{G}}_{\ell+1} \text{\ \ or\ \ } \mathcal{\tilde{G}}_{\ell+1}\cap [M_{\ell+1}]^{<\omega}\subset S_\alpha.$$  

For the same reason as before, the first containment cannot hold.  Thus we have a decreasing sequence $(M_n)\subset [\mathbb{N}]$ so that $1<\min M_1<\min M_2<\ldots$ and $\mathcal{\tilde{G}}_n\cap [M_n]^{<\omega}\subset S_\alpha$ for each $n$.  We let $m_0=1$, $m_n=\min M_n$, and $M=(m_n)_{n\geq 0}$.  Note that $(m_i)_{i\geq n}\subset M_n$ for each $n$.  

Fix a sequence $\overline{\delta}=(\delta_n)\subset (0,1)$ so that for each $n$, \begin{equation}3\delta_n<\min \{\varepsilon_{n,n}, \rho^{-n-1}\}, \text{\ and}\label{2}\end{equation} \begin{equation} 3\displaystyle\sum_{n=1}^\infty \delta_n <C-D.\label{3}\end{equation}

Suppose $(x_n)$ is a $\overline{\delta}$-block sequence in the blocked FDD $G$ defined as in the statement of the theorem using the chosen $m_n$, and $1\leq s_1<s_2<\ldots $ is such that $\|P^G_{[s_{n-1},s_m)}x_n-x_n\|<\delta_n$.  

Define $$z_n=\frac{P^H_{[s_{n-1},s_n)}x_n}{\|P^H_{[s_{n-1},s_n)}x_n\|}.$$  

It follows from this definition that $\|z_n-x_n\|<2\delta_n$.  Let $(w_n)$ be a normalized block sequence so that $\text{supp}_H w_n\subset \text{supp}_H z_n$, $\|z_n-w_n\|<\delta_n$, and $\min \text{supp}_E w_n=m_{s_{n-1}}$.  

Then $\|x_n-w_n\|<3\delta_n$ for each $n$.  From $(5.3)$, it suffices to prove that $(w_n)$ is $D$-dominated by $(e_{m_{s_{n-1}}})$ to show that $(x_n)$ is $C$-dominated by $(e_{m_{s_{n-1}}})$.  

Fix $a=(a_n)\in c_{00}$.  Let $w^*\in S_{Z^*}$ be such that $w^*\Bigl(\displaystyle\sum_{n=1}^\infty a_nw_n\Bigr)=\Bigl\|\displaystyle\sum_{n=1}^\infty a_nw_n\Bigr\|$.  For any $F\subset \mathbb{N}$, we let $ms(F)=\{m_{s_{n-1}}:n\in F\}$.  For each $j$, let \begin{align*} I^+_j & =\{n\in \text{supp}(a):n<j, \rho^j< w^*(w_n)\leq \rho^{j-1}\}, \\ I^-_j & =\{n\in \text{supp}(a): n<j, \rho^j<-w^*(w_n)\leq \rho^{j-1}\}, \\ J^+_j&=\{n\in \text{supp}(a): n\geq j, \rho^j<w^*(w_n)\leq \rho^{j-1}\}, \\ J^-_j&=\{n\in \text{supp}(a): n\geq j, \rho^j<-w^*(w_n)\leq \rho^{j-1}\}.\end{align*}

We will prove that $ms(J_j^{\pm})\in S_\alpha$ for each $j$.  We note that $s_{n-1}\geq n$, which means $$ms(J^{\pm}_j)=(m_{s_{n-1}})_{n\in J^{\pm}_j}\subset (m_n)_{n\geq j}\subset M_j.$$

We will show that $(w_n)_{n\in ms(J^+_j)}\in \mathcal{G}_j =\Sigma(E,Z)\cap (\mathcal{F}_j)^Z_{\overline{\varepsilon}_n}$.  Containment in $\Sigma(E,Z)$ is clear.  For each $n\in ms(J_j^+)$, $$w^*(x_n)\geq w^*(w_n)-w^*(w_n-x_n)>\rho^j-3\delta_j\geq \rho^j-\rho^{j+1}>2\rho^{j+1}.$$  

Here, we use the definition of $J^+_j$ and the fact that $\rho<\frac{1}{3}$.  By the geometric version of the Hahn-Banach Theorem, the existence of such a $w^*\in B_{Z^*}$ is sufficient to give that $(x_n)_{n\in J^+_j}\in \mathcal{F}_j$.  

Since $\min J^+_j\geq j$, $n\in J^+_j$, $$\|x_n-w_n\|<3\delta_n\leq \varepsilon_{n,n}\leq \varepsilon_{j,n}.$$

Thus $(w_n)_{n\in J^{\pm}_j}$ is a $\overline{\varepsilon}_j$ perturbation of $(x_n)_{n\in J^+_j}$, hence $(w_n)_{n\in J^+_j}\in \mathcal{G}_j$.  This means $ms(J^+_j)\in \mathcal{\tilde{G}}_j$.  Combining these results yields $$ms(J^+_j)\in \mathcal{\tilde{G}}_j\cap[M_j]^{<\omega}\subset S_\alpha.$$

A similar argument using $-w^*$ gives that $ms(J^-_j)\in S_\alpha$.  

We note that \begin{align*} \displaystyle\sum_{n\in J^{\pm}_j} a_nw^*(w_n)& \leq \rho^{j-1}\displaystyle\sum_{n\in J^{\pm}_j} |a_n| \\ &=\rho^{j-1} \Bigl\|\displaystyle\sum_{n\in J^{\pm}_j}a_ne_{m_{s_{n-1}}}\Bigr\|_{X_\alpha}\leq \rho^{j-1}\Bigl\|\displaystyle\sum_{n=1}^\infty a_n e_{m_{s_{n-1}}}\Bigr\|_{X_\alpha}.\end{align*}

By $1$-unconditionality,  $|a_k|\leq \Bigl\|\displaystyle\sum_{n=1}^\infty a_n e_{m_{s_{n-1}}}\Bigr\|_{X_\alpha}$.  Because $|I^{\pm}_j|<j$, it follows that $$\displaystyle\sum_{n\in I^{\pm}_j}a_nw^*(w_n)\leq \rho^{j-1}(j-1)\Bigl\|\displaystyle\sum_{n=1}^\infty a_ne_{m_{s_{n-1}}}\Bigr\|_{X_\alpha}.$$

It follows that \begin{align*} \Bigl\|\displaystyle\sum_{n=1}^\infty a_nw_n\Bigr\| & =\displaystyle\sum_{j=1}^\infty \displaystyle\sum_{n\in I^+_j}a_nw^*(w_n)+\displaystyle\sum_{j=1}^\infty \displaystyle\sum_{n\in I^-_j}a_nw^*(w_n) \\ & +\displaystyle\sum_{j=1}^\infty \displaystyle\sum_{n\in J^+_j}a_n w^*(w_n)+\displaystyle\sum_{j=1}^\infty \displaystyle\sum_{n\in J^-_j} a_nw^*(w_n) \\ &\leq \Bigl\|\displaystyle\sum_{n=1}^\infty a_ne_{m_{s_{n-1}}}\Bigr\|_{X_\alpha} \displaystyle\sum_{j=1}^\infty \Bigl(2(j-1)\rho^{j-1}+2\rho^{j-1} \Bigr) \\ & =2\Bigl\|\displaystyle\sum_{n=1}^\infty a_ne_{m_{s_{n-1}}}\Bigr\|_{X_\alpha}\displaystyle\sum_{j=1}^\infty j\rho^{j-1} \\ &=\frac{2}{(1-\rho)^2}\Bigl\|\displaystyle\sum_{n=1}^\infty a_ne_{m_{s_{n-1}}}\Bigr\|_{X_\alpha}<D\Bigl\|\displaystyle\sum_{n=1}^\infty a_n e_{m_{s_{n-1}}}\Bigr\|_{X_\alpha}.  \end{align*}

This implies the desired conclusion.

\end{proof}

We are now ready to prove Theorem $1.1$

\begin{proof}[Proof of Theorem 1.1]

Because $X$ has countable Szlenk index, $X^*$ must be separable.  By a theorem of Zippin \cite{Z}, $X$ embeds into a space $Z$ with shrinking, bimonotone FDD $E$.  By renorming $X$ with an equivalent norm, we can assume $X$ is isometrically a subspace of $Z$.  Fix $C>2$ and take $M=(m_n)_{n\geq 0}$ and $\overline{\delta}$ given in Theorem $5.5$, and let $H$ be the corresponding blocking.    

Take a normalized, weakly null even tree $(x_t)_{t\in T^\text{even}_\infty}$.  Put $s_0=1$, $k_1=1$.  Next, assume $s_0<s_1<\ldots<s_{\ell-1}$ and $n_1<\ldots< n_{2\ell-1}$ have been chosen so that $$\|P^H_{[s_{n-1},s_n)}x_{(k_1, \ldots, k_{2n})}-x_{(k_1, \ldots, k_{2n})}\|<\delta_n$$ for each $n< \ell$.  

Because nodes are weakly null, there exists $k_{2\ell}>k_{2\ell-1}$ so that $$\|P^H_{[1,s_{\ell-1})}x_{(k_1, \ldots, k_{2\ell})}\|<\delta_{\ell}.$$

Next, choose $s_{\ell}>s_{\ell-1}$ so that $$\|P^H_{[s_{\ell-1}, s_{\ell})}x_{(k_1, \ldots, k_{2\ell})}-x_{(k_1, \ldots, k_{2\ell})}\|<\delta_{\ell}.$$

Finally, choose $k_{2\ell+1}>\max \{m_{s_{\ell}}, k_{2\ell}\}$.  

We deduce that $(k_{2n-1},x_{(k_1, \ldots, k_{2n})})_{n=1}^\infty$ is $C$ dominated by $(e_{m_{s_{n-1}}})$.  Since $m_{s_{n-1}}<k_{2n-1}$ and the Schreier spaces are $1$-right dominant, the branch $(k_{2n-1},x_{(k_1, \ldots, k_{2n})})$ is $C$ dominated by $(e_{k_{2n-1}})$.  Thus $X$ has subsequential $X_\alpha$-upper tree estimates.   

\end{proof}

The following corollary proves Theorem $1.2$ and Corollary $1.3$.  

\begin{corollary} Let $\alpha$ be a countable ordinal.  There exists a Banach space $Z$ with bimonotone, shrinking FDD $E$ which satisfies subsequential $X_\alpha$-upper block estimates in $Z$ which is universal for the class $\bigl \{X: Sz(X)\leq \omega^\alpha\bigr\}$.  Moreover, there exists a Banach space $W$ with a basis such that $Sz(W)\leq \omega^{\alpha+1}$ which is also universal for this class.  \end{corollary}

\begin{proof}
Let $Z$ be the universal space for the class $\mathcal{A}_{X_\alpha}$ guaranteed by Theorem $5.4$, and let $E$ be its FDD.  From the proof of Theorem $5.4$, we see that $E$ satisfies subsequential $X_\alpha$-upper block estimates in $Z$.  If $X$ is a Banach space such that $Sz(X)\leq \omega^{\alpha}$, then $X^*$ is separable \cite{SZLENK}.  By Corollary $5.6$, $X$ satisfies subsequential $X_\alpha$-upper tree estimates.  By the definition of $\mathcal{A}_{X_\alpha}$ and choice of $Z$, $X$ embeds into $Z$.   By Corollary $4.5$ and Proposition $4.6$, $$Sz(Z)\leq Sz(X_\alpha)=\omega^{\alpha+1}.$$  By \cite[Corollary 4.12]{JRZ}, there exists a sequence of finite dimensional spaces $(H_n)$ so that if $D=\Bigl(\underset{n=1}{\overset{\infty}{\oplus}} H_n\Bigr)_2$, then $W=Z\oplus D$ has a basis.  Since the FDD $(H_n)$ satisfies $\ell_2$-upper block estimates in $D$, $Sz(D)\leq \omega$ \cite[Theorem 3]{OS}.  By \cite[Proposition 14]{OSZ2}, $$Sz(W)= \max \{Sz(Z), Sz(D)\}\leq \omega^{\alpha+1}.$$

\end{proof}

\begin{remark}
We summarize what we have shown.  We have established that if $\alpha<\omega_1$, then $$\bigl\{X:Sz(X)\leq \omega^\alpha\bigr\}\subsetneq \mathcal{A}_{X_\alpha}\subset \bigl\{X: Sz(X)\leq \omega^{\alpha+1}\bigr\}.$$

The first inclusion comes from Corollary $5.6$.  The strict inclusion comes by noting that $X_\alpha$ satisfies subsequential $X_\alpha$-upper block estimates but has Szlenk index $\omega^{\alpha+1}$.  The second inclusion is a consequence of Corollary $4.5$ and Proposition $4.6$.  

\end{remark}

\section{Applications}

 \begin{definition}
For Banach spaces $X,Y$, we consider $X\otimes Y$ as a space of bounded operators from $Y^*$ into $X$, endowed with the topology induced by the operator norm.  For each expression $\displaystyle\sum_{n=1}^\ell x_n\otimes y_n$, we define $$\Bigl(\displaystyle\sum_{n=1}^\ell x_n\otimes y_n\Bigr)(y^*)=\displaystyle\sum_{n=1}^\ell y^*(y_n)x_n\text{\ \ \ \  } x_n\in X, y_n\in Y, y^*\in Y^*.$$

We denote by $X\otimes Y$ the space of equivalence classes of all such expressions, where two expressions are equivalent if they determine the same operator.

\end{definition}

Since such operators are finite rank, they are compact.  Thus the completion of the injective product, denoted $X\hat{\otimes}_\epsilon Y$, must be contained within the compact operators, $K(Y^*,X)$.  

It is easy to verify that $\Bigl(\displaystyle\sum_{n=1}^\ell x_n\otimes y_n\Bigr)^* = \displaystyle\sum_{n=1}^\ell y_n\otimes x_n\in Y\otimes _\epsilon X$.  Thus, via adjoints, $X\otimes _\epsilon Y$ is isometrically isomorphic to $Y\otimes _\epsilon X$, and the same is true of the completions.  This means that if $u\in X\hat{\otimes}_\epsilon Y$, $u^*\in K(X^*, Y)$.  

\begin{definition}

A Banach space $X$ is said to have the \emph{approximation property} if, for any $C\subset X$ compact and $\varepsilon>0$, there exists a bounded, finite rank operator $T:X\to X$ such that $\|Tx-x\|<\varepsilon$ for all $x\in C$.  

\end{definition}

If either $X$ or $Y$ has the approximation property, any element $u\in K(X,Y)$ is the limit of bounded, finite rank operators.  Since any space with an FDD has the approximation property, if $E$ has FDD $(E_i)$ and $u\in K(X,E)$, for some Banach space $X$, $P_n^E u\to u$ in norm.  
\vspace{10mm}

\begin{proposition}
Let $V$ be a Banach space with normalized, $1$-unconditional basis $(e_n)$.  Let $X,E$ be Banach spaces, $E$ with FDD $(E_n)$ satisfying subsequential $C$-$V$-upper block estimates.  Let $u_n:X\to E$ be bounded operators and $1=k_0<k_1\ldots<k_\ell$ natural numbers such that $u_n(X)\subset \underset{j=k_{n-1}}{\overset{k_n-1}{\bigoplus}}E_j$.  Then  $$\Bigl\|\displaystyle\sum_{n=1}^\ell u_n\Bigr\|\leq C\Bigl\|\displaystyle\sum_{n=1}^\ell \|u_n\|e_{k_{n-1}}\Bigr\|.$$

\end{proposition}

\begin{proof} 

Let $u=\displaystyle\sum_{n=1}^\ell u_n$.  Take $x\in B_X$.  Let $N=\{n\leq \ell: u_n(x)\neq 0\}$.  If this set is empty, then $u(x)=0$.  Otherwise, $(u_n(x))_{n\in N}$ is a block sequence in $E$.  Let $m_n=\min \text{supp\ }u_n(x)$.  By Proposition $2.4$, we get that $$\|u(x)\|_Z\leq C\Bigl\|\displaystyle\sum_{n\in N} \|u_n(x)\|_Ze_{k_{n-1}}\Bigr\|_V\leq \Bigl\|\displaystyle\sum_{n=1}^\ell \|u_n\|e_{k_{n-1}}\Bigr\|.$$

Since this holds for any $x\in B_X$, we get the result.  

\end{proof}

\begin{definition}

Let $E,F$ be Banach spaces with shrinking, bimonotone FDDs $(E_n), (F_n)$.  Then let $$H_n=\text{span}(E_i\otimes_\epsilon F_j: \max\{i,j\}=n).$$

We call this \emph{the square blocking}.  \end{definition}

\begin{proposition} If $W,Z$ are Banach spaces with FDDs $(E_n), (F_n)$, then $(H_n)$ is an FDD for $W\hat{\otimes}_\epsilon Z$.  If $(E_n),(F_n)$ are shrinking, so is $(H_n)$.  \end{proposition}

\begin{proof}

Let $P_A=P^E_A$ and $Q_A=P^F_A$ denote the projections in $E,F$, respectively.  Then $P^H_n:E\hat{\otimes}_\epsilon F\to H_n$ is defined by $P^H_n(u)=P_{[1,n]}uQ_{[1,n]}^*-P_{[1,n)}uQ_{[1,n)}^*$, where $P_\varnothing=Q_\varnothing=0$.  This means $P^H_{[1,n]}:E\hat{\otimes}_\epsilon F\to \underset{i=1}{\overset{n}{\bigoplus}}H_i$ is given by $P^H_{[1,n]}(u)=P_{[1,n]}uQ_{[1,n]}^*$.  

Since $u:F^*\to E$ is compact, $P_{[1,n]}u\underset{n\to\infty}{\to} u$ in norm.  Moreover, since $u^*:E^*\to F$ is compact, $Q_{[1,n]}u^*\underset{n\to\infty}{\to} u^*$.  But this means that $uQ_{[1,n]}^*\underset{n\to\infty}{\to} u$.  So \begin{align*} \|u-P_{[1,n]}uQ_{[1,n]}^*\| & \leq \|u-P_{[1,n]}u\|+\|P_{[1,n]}u-P_{[1,n]}uQ_{[1,n]}^*\| \\ & \leq \|u-P_{[1,n]}u\|+\|u-uQ_{[1,n]}^*\|\to 0.\end{align*}

Thus $P^H_{[1,n]}(u)\to u$.  Moreover, if $u_m\in H_m$ then $P_nu_mQ_n^*-P_{n-1}u_mQ_{n-1}^*=\delta_{mn}u_m$.  Thus if $u=\displaystyle\sum_{m=1}^\infty u_m$, $u_m\in H_m$, then $u_n=P_nuQ_n^*$, and we have uniqueness.  So $(H_n)$ is an FDD of $E\hat{\otimes}_\epsilon F$.  

\vspace{5mm}

The FDD $(H_n)$ is shrinking if any sequence $(x_n)\subset B_{E\hat{\otimes}_\epsilon F}$ such that $P^H_{[1,n]}x_n=P_{[1,n]}x_nQ_{[1,n]}^*=0$ is weakly null.  A sequence $(x_n)\subset E\hat{\otimes}_\epsilon F$ is weakly null if and only if for any $g^*\in E^*$ and $f^*\in F^*$, $g^*\otimes f^*(x_n)\to 0$ (Lemma 1.1, [L]).  Here, $g^*\otimes f^*(x)=g^*(x(f^*))$.    

But $$x_n=P_{[1,n]} x_nQ_{[1,n]}^*+P_{[1,n]}x_nQ_{(n,\infty)}^*+P_{(n,\infty)}x_n$$ $$=P_{[1,n]}x_nQ_{(n,\infty)}^*+P_{(n,\infty)}x_n.$$

Take $f^*\in F^*$ and $g^*\in E$, $$g^*\Bigl(P_{[1,n]}x_nQ_{(n,\infty)}^*f^*\Bigr)\leq \|g^*\|\|Q_{(n,\infty)}^*f^*\|\to 0$$

because $(F_n)$ is shrinking.  Thus $P_{[1,n]}x_nQ^*_{(n,\infty)}f^*$ is weakly null in $E$, and $P_{[1,n]}x_nQ^*_{(n,\infty)}$ is weakly null in $E\hat{\otimes}_\epsilon F$.  A similar argument shows that $P_{(n,\infty)}x_n$ is weakly null.  Thus $x_n= P_{[1,n]}x_nQ_{(n,\infty)}^*+P_{(n,\infty)}x_n$ is weakly null.  This means that $(H_n)$ is a shrinking FDD.  

\end{proof}

\begin{lemma}
Let $V$ be a Banach space with normalized, $1$-unconditional basis $(v_n)$.  Let $W,Z$ be Banach spaces with shrinking, bimonotone FDDs $(E_n)$, $(F_n)$ satisfying subsequential $C$-$V$-upper block estimates.  Then $W\hat{\otimes}_\epsilon Z$ with FDD $(H_n)$ satisfies subsequential $2C$-$V$-upper block estimates.

\end{lemma}

\begin{proof}

Take a normalized sequence $(u_n)$ in $W\hat{\otimes}_\epsilon Z$ which is a block sequence with respect to $(H_n)$.  Let $m_n=\min \text{supp\ } u_n$.  Then $u_n=P^E_{[1,m_{n+1})}u_nP^{F^*}_{[1,m_{n+1})}$ and $0=P^E_{[1,m_n)}u_nP^{F^*}_{[1,m_n)}$.  Let $$a_n=P^E_{[m_n,m_{n+1})}u_nP^{F^*}_{[1,m_n)},$$ $$b_1=0,$$ and $$b_n=u_nP^{F^*}_{[m_n,m_{n+1})}.$$  

By construction, $a_n+b_n=u_n$ for all $n$.  The bimonotonicity of the FDDs gives that $\|a_n\|, \|b_n\|\leq 1$ for each $n$.  Let $N_1=\{n: a_n\neq 0\}$, $N_2=\{n:b_n\neq 0\}$.  We note that for $n\in N_2$, the adjoint of $b_n$ satisfies $b_n^*=P^F_{[m_n,m_{n+1})}u_n^*\neq 0$.  Moreover, $(a_n)_{n\in N_1}$, $(b_n^*)_{n\in N_2}$ satisfy the hypotheses of Proposition $6.1$ as operators from $Z^*$ to $W$ and from $W^*$ to $Z$, respectively.  This means that for any $(c_n)\subset \mathbb{R}$, $$\Bigl\|\displaystyle\sum_{n\in N_1}c_na_n\Bigr\|\leq C\Bigl\|\displaystyle\sum_{n\in N_1}c_n \|a_n\|v_{m_n}\Bigr\|\leq C\Bigl\|\displaystyle\sum_{n=1}^\infty a_n v_{m_n}\Bigr\|.$$

Similarly, $$\Bigl\|\displaystyle\sum_{n\in N_2} c_nb_n\Bigr\|=\Bigl\|\displaystyle\sum_{n\in N_2}c_n b^*_n\Bigr\|\leq C\Bigl\|\displaystyle\sum_{n=1}^\infty a_n v_{m_n}\Bigr\|$$

Then $$\Bigl\|\displaystyle\sum_{n=1}^\infty c_n u_n\Bigr\|\leq \Bigl\|\displaystyle\sum_{n\in N_1} c_n a_n\Bigr\|+\Bigl\|\displaystyle\sum_{n\in N_2} c_nb_n\Bigr\|\leq 2C \Bigl\|\displaystyle\sum_{n=1}^\infty c_n v_{m_n}\Bigr\|.$$

\end{proof}

\begin{theorem}

Let $X,Y$ be nonzero Banach spaces with seprable duals.  If either space has finite dimension, then $Sz(X \hat{\otimes}_\epsilon Y)=\max \{Sz(X), Sz(Y)\}$.  Otherwise, let $\beta<\omega_1$ be such that $\max \{Sz(X),Sz(Y)\}=\omega^\beta$.  Then $$Sz(X\hat{\otimes}_\epsilon Y)\leq \omega^{\beta+1}.$$  If $\beta=1$ or $\beta=\alpha \omega$ for some $\alpha<\omega_1$, then $Sz(X\hat{\otimes}_\epsilon Y)=\omega^\beta$.  

\end{theorem}

\begin{proof}

Since both $X$ and $Y$ embed into $X\hat{\otimes}_\epsilon Y$, $\max\{Sz(X),Sz(Y)\}\leq Sz(X\hat{\otimes}_\epsilon Y)$.  

Consider the case that $0< n=\dim X<\infty$.  Then $X$ is isomorphically $\ell_\infty^n$.  This means $$X\hat{\otimes}_\epsilon Y=\ell^n_\infty \hat{\otimes}_\epsilon Y=\Bigl(\underset{i=1}{\overset{n}{\oplus}} Y\Bigr)_\infty.$$

By \cite[Proposition 14]{OSZ2}, $\Bigl(\underset{i=1}{\overset{n}{\oplus}} Y\Bigr)_\infty=\max\{Sz(X), Sz(Y)\}=Sz(Y)$.  

Assume both spaces have infinite dimension.  If $\beta=1$, then by \cite[Theorem 3]{OS} there exists some $q>1$ so that $X,Y$ satisfy subsequential $\ell_q$-upper tree estimates.  In this case, put $V=\ell_q$.  If $\beta=\alpha\omega$, then by \cite{FOSZ} there exists some $c\in (0,1)$ so that $X,Y$ satisfy subsequential $T_{\alpha,c}$-upper tree estimates.  Here, $T_{\alpha, c}$ is the Tsirelson space of order $\alpha$.  In this case, put $V=T_{\alpha,c}$.  If we are not in one of these two cases, $X,Y$ satisfy subsequential $X_\beta$-upper tree estimates, and we let $V=X_\beta$.

By Theorem $1.1$, there exist spaces $W,Z$ with shrinking, bimonotone FDDs $E,F$, respectively, which satisfy subsequential $V$-upper block estimates and so that $X,Y$ embed in $W,Z$, respectively.  Because injective tensor products respect subspaces, $X\hat{\otimes}_\epsilon Y\hookrightarrow W\hat{\otimes}_\epsilon Z$.  Thus $Sz(X\hat{\otimes}_\epsilon Y)\leq Sz(W\hat{\otimes}_\epsilon Z)$.  By Lemma $6.3$, $W\hat{\otimes}_\epsilon Z$ satisfies subsequential $V$-upper block estimates.  By Corollary $4.5$, $Sz(W\hat{\otimes}_\epsilon Z)\leq Sz(V)$.  Since $Sz(\ell_q)=\omega$, $Sz(T_{\alpha,c})\linebreak=\omega^{\alpha \omega}$ \cite[Proposition 16]{OSZ2}, and $Sz(X_\beta)=\omega^{\beta+1}$, we have the result.

\end{proof}

\subsection*{Acknowledgements}
The research of the author was supported by the National Science Foundation grant DMS0856148


\begin{thebibliography}{HD}

\normalsize
\baselineskip=17pt

\bibitem{AJO} D. Alspach, R. Judd, E. Odell. \emph{The Szlenk index and local $\ell_1$-indices}, Positivity, 9 (2005), no. 1,1-44. 

\bibitem{DFJP} W. J. Davis, T. Figiel, W. B. Johnson, A. Pelczynski, \emph{Factoring weakly compact operators}, Journal of Functional Analysis, 17. No. 3 (1974), 311-327.  

\bibitem{FOSZ} D. Freeman, E. Odell, Th. Schlumprecht, A. Zs\'{a}k, \emph{Banach spaces of bounded Szlenk index II}, Fund. Math. 205 (2009) 161-177.

\bibitem{G} I. Gasparis, \emph{A dichotomy theorem for subsets of the power subsets of the power set of the natural numbers}, Proceedings of the American Mathematical Society, 129 (2001) 759-764. 

\bibitem{J} W. B. Johnson, \emph{On quotients of $L_p$ which are quotients of $\ell_p$}, Compositio Math. Vol. 34, Fasc. (1977), 69-89.  

\bibitem{JRZ} W. B. Johnson, H. Rosenthal, M. Zippin, \emph{On bases, finite dimensional decompositions, and weaker structures in Banach spaces}, Israel J. Math. {\bf 9} (1971), 488-506.

\bibitem{JZ} W. B. Johnson, M. Zippin, \emph{On subapces of quotients of $\bigl(\Sigma G_n\bigr)_{\ell_p}$ and $\bigl(\Sigma G_n\bigr)_{c_0}$}, Israel J. Math. 13 Nos. 3-4, (1972).  

\bibitem{L} D. R. Lewis, \emph{Conditional weak compactness in certain inductive tensor products}. Math. Ann. 201 (1973) 201-209.  

\bibitem{La} G. Lancien, \emph{A survey on the Szlenk index and some of its applications,} Rev. R. Acad. Cienc. Exactas Fis. Nat. Ser. A Mat. 100 (1-2), (2006) 209-235.

\bibitem{OS1} E. Odell and Th. Schlumprecht, \emph{Trees and branches in Banach spaces}, Trans. Amer. Math. Soc. 354 no. 10 (2002) 4085-4108.  

\bibitem{OS} E. Odell and Th. Schlumprecht, \emph{Embedding into Banach spaces with finite dimensional decompositions}, Rev. R. Acad. Cienc. Exactas Fis. Nat. Ser. A Mat. vol 100 (1-2)(2006), 295-323.

\bibitem{OSZ1} E. Odell, Th. Schlumprecht, A. Zs\'{a}k, \emph{A new infinite game in Banach spaces with applications,} Banach Spaces and their Applications in Analysis, pp. 147-182, Walter de Gruyter, Berlin, 2007.  

\bibitem{OSZ2} E. Odell, Th. Schlumprecht, A. Zs\'{a}k.  \emph{Banach spaces of bounded Szlenk index}, Studia Math. 183 (2007), no. 1, 63-97. 

\bibitem{P} A. Pe\l czy\'{n}ski, \emph{Universal bases}, Studia Math. 32 (1969), 247-268.  

\bibitem{Sch} G. Schechtman. \emph{On Pelczynski's paper ``Universal bases,"} Israel J. Math. 22 (1975) no. 3-4, 181-184.  

\bibitem{SZLENK} W. Szlenk, \emph{The non existence of a separable reflexive Banach space universal for all separable reflexive Banach spaces}, Studia Math. 30 (1968),53-61. 

\bibitem{Z} M. Zippin, \emph{Banach spaces with separable duals}, Trans. Amer. Math. Soc. 310 (1988), no. 1, 371-379. 





\end{thebibliography}
\end{document}